\documentclass[12pt]{amsart}

\usepackage{mathrsfs,amssymb}

\usepackage{booktabs}
\newtheorem{theorem}{Theorem}[section]
\newtheorem{lemma}[theorem]{Lemma}
\newtheorem{prop}[theorem]{Proposition}
\newtheorem{cor}[theorem]{Corollary}

\theoremstyle{definition}

\newtheorem{con}[theorem]{Conjecture}

\theoremstyle{remark}
\newtheorem{remark}[theorem]{Remark}

\numberwithin{equation}{section}

\let \la=\lambda
\let \e=\varepsilon
\let \d=\delta
\let \o=\omega
\let \a=\alpha
\let \f=\varphi
\let \b=\beta

\let \O=\Omega
\let \si=\sigma

\let \ga=\gamma

\begin{document}
\title[On two weight estimates]
{On two weight estimates for iterated commutators}

\author[A. K. Lerner]{Andrei K. Lerner}

\address[A. K. Lerner]{Department of Mathematics,
Bar-Ilan University, 5290002 Ramat Gan, Israel}
\email{lernera@math.biu.ac.il}

\thanks{The first author was supported by ISF grant No. 447/16 and ERC Starting Grant No. 713927.
The second and third authors were supported by CONICET PIP 11220130100329CO and  ANPCyT PICT 2018-02501.}

\author[S. Ombrosi]{Sheldy Ombrosi}
\address[S. Ombrosi]{Departamento de Matem\'atica\\
Universidad Nacional del Sur\\
Bah\'ia Blanca, 8000, Argentina}\email{sombrosi@uns.edu.ar}

\author[I. P. Rivera-R\'{\i}os]{Israel P. Rivera-R\'{\i}os}
\address[I. P. Rivera-R\'{\i}os]{Departamento de Matem\'atica\\
Universidad Nacional del Sur\\
Bah\'ia Blanca, 8000, Argentina}
\email{israel.rivera@uns.edu.ar}

\begin{abstract}
In this paper we extend the bump conjecture and
a particular case of the separated bump conjecture with logarithmic bumps to iterated commutators~$T_b^m$.
Our results are new even for the first order commutator $T_b^1$. A new bump type necessary condition for the two-weighted boundedness of $T_b^m$
is obtained as well. We also provide some results related to a converse to Bloom's theorem.
\end{abstract}



\maketitle

\section{Introduction}
Let $T$ be a Calder\'on-Zygmund operator, and let $b$ be a locally integrable function on ${\mathbb R}^n$. The commutator $[b,T]$ of $T$ and $b$ is
defined by
$$[b,T]f(x)=b(x)T(f)(x)-T(bf)(x).$$
The iterated commutators $T_b^m, m\in {\mathbb N},$ are defined inductively by
$$T_b^mf=[b,T_b^{m-1}]f,\quad T_b^1f=[b,T]f.$$

By a weight we mean a non-negative, locally integrable function. In this paper we study two weight estimates
$$\int_{{\mathbb R}^n}|T_b^mf|^pu\lesssim \int_{{\mathbb R}^n}|f|^pv\quad(p>1)$$
with emphasis on $A_p$ type conditions on a couple of weights $(u,v)$.

This subject has a long history, and in our brief exposition below we mention only several papers of specific interest to us.
Consider first a Calder\'on-Zygmund operator $T$ (which formally can be regarded as $T_b^0$). In the one-weighted case when $u=v=w$,
the $A_p$ condition,
$$\sup_Q\left(\frac{1}{|Q|}\int_Qw\right)\left(\frac{1}{|Q|}\int_Qw^{-\frac{1}{p-1}}\right)^{p-1}<\infty,$$
is the sufficient (and for a subclass of non-degenerate Calder\'on-Zygmund operators is also necessary) condition for $T$ to be bounded on $L^p(w)$ (see, e.g., \cite{S}).

It is well known that the two weight problem is much more complicated, and, in particular, the $A_p$ condition for a couple $(u,v)$,
\begin{equation}\label{apcoup}
\sup_Q\left(\frac{1}{|Q|}\int_Qu\right)\left(\frac{1}{|Q|}\int_Qv^{-\frac{1}{p-1}}\right)^{p-1}<\infty,
\end{equation}
is no longer sufficient for $T: L^p(v)\to L^p(u)$.

There was a great deal of effort to find slightly stronger bump conditions which are sufficient for $T: L^p(v)\to L^p(u)$.
To formulate such conditions, define the normalized Luxemburg norm (for a Young function $\f$) by
$$
\|f\|_{\f,Q}=\inf\Big\{\la>0:\frac{1}{|Q|}\int_Q\f(|f(y)|/\la)dy\le 1\Big\}.
$$
If $\f(t)=t^p\log^{\a}(e+t),\a\ge 0$, we will use the notation $\|f\|_{L^p(\log L)^{\a},Q}$.
Observe that in this notation, (\ref{apcoup}) can be written in the form
$$\sup_{Q}\|u^{1/p}\|_{L^p,Q}\|v^{-1/p}\|_{L^{p'},Q}<\infty,$$
where $\frac{1}{p}+\frac{1}{p'}=1$. The bump conditions strengthen this condition by replacing the $L^p$ norms by slightly larger Luxemburg norms.

We say that a Young function $A$ satisfies the $B_p$ condition if $\int_1^{\infty}\frac{A(t)}{t^p}\frac{dt}{t}<\infty$. Let $\bar A$ denote the Young
function complementary to $A$. The bump conjecture of D.~Cruz-Uribe and C. P\'erez (see \cite[p. 187]{CMP1}) says that if $A$ and $B$ are Young functions such that $\bar A\in B_{p'}$ and $\bar B\in B_p, p>1,$
and
\begin{equation}\label{bumcon}
\sup_{Q}\|u^{1/p}\|_{A,Q}\|v^{-1/p}\|_{B,Q}<\infty,
\end{equation}
then $T: L^p(v)\to L^p(u)$. The bump conjecture was solved positively in \cite{NRTV} for $p=2$ and in~\cite{L} for all $p>1$.

Observe that typical examples of $A$ and $B$ satisfying $\bar A\in B_{p'}$ and $\bar B\in B_p$ are
\begin{equation}\label{typex}
A(t)=t^p\log^{p-1+\d}(e+t)\quad\text{and}\quad B(t)=t^{p'}\log^{p'-1+\d}(e+t),
\end{equation}
where $\d>0$. Such functions are called the logarithmic bumps.

The separated bump conjecture (probably first formulated in \cite{CRV}) asserts that $T$ is bounded from $L^p(v)$ to $L^p(u)$ when (\ref{bumcon}) is replaced by
\begin{equation}\label{sepbc}
\sup_{Q}\|u^{1/p}\|_{L^p,Q}\|v^{-1/p}\|_{B,Q}<\infty\quad\text{and}\quad \sup_{Q}\|u^{1/p}\|_{A,Q}\|v^{-1/p}\|_{L^{p'},Q}<\infty.
\end{equation}
This conjecture is still open, in general. In \cite{CRV}, D. Cruz-Uribe, A. Reznikov and A.~Volberg established that in the particular case of $A$ and $B$ in (\ref{typex}), this conjecture is true
(see also \cite{HP} for a different proof of this result).

We also mention the works of M. Lacey \cite{La} and K. Li \cite{Li} where some different variants of the separated bump conjecture were obtained
(which in the particular case of the logarithmic bumps in (\ref{typex}) provide yet another approach to the result in \cite{CRV}). In Section 4.2 below a more detailed exposition of the results in \cite{La,Li} is given.

Turn now to the commutators $T_b^m$. Our first goal is to extend the bump conjectures to $T_b^m$.
It was shown in our previous works \cite{LOR1, LOR2} (we recall the proof in Lemma \ref{dual1})
that $T_b^m$ (for $b\in BMO$) is controlled by the $(m+1)$-th iteration of $A_{\mathcal S}$, denoted by $A_{\mathcal S}^{m+1}$, where
$A_{\mathcal S}$ is the standard sparse operator defined by
$$A_{\mathcal S}f(x)=\sum_{Q\in {\mathcal S}}f_Q\chi_Q(x)\quad(f_Q=\frac{1}{|Q|}\int_Qf).$$
Observe that in the case $m=0$ this result is well known, see \cite{CR1,HRT,La1,L1,LN,LO}.

A domination of $T$ by $A_{\mathcal S}$ is the standard tool in most of the works dealing with the bump conjectures for $T$. Therefore, it is not surprising that in our attempt
of extending these results to $T_b^m$ we deal with $A_{\mathcal S}^{m+1}$. We will show (Lemma \ref{dual2}) that $A_{\mathcal S}^{m+1}$ is essentially equivalent to $T_m+T_m^*$, where  $T_m$ is a positive linear operator controlled by
$$A_{L(\log L)^m,{\mathcal S}}f(x)=\sum_{Q\in {\mathcal S}}\|f\|_{L(\log L)^m,Q}\chi_Q(x).$$
Thus, the operator $A_{L(\log L)^m,{\mathcal S}}$ is the key object in our analysis. The $L^p(v)\to L^p(u)$ bounds for $T_b^m$ follow from
the corresponding bounds for $A_{L(\log L)^m,{\mathcal S}}$ and their dual counterpart.

In what follows, it will be convenient to use the notation
$$[\la,\mu]_{A,B}=\sup_{Q}\|\la\|_{A,Q}\|\mu\|_{B,Q}.$$
Our extension of the bump conjecture to $T_b^m$ is the following.

\begin{theorem}\label{extbctbm}
Let ${\mathcal S}$ be a sparse family. Assume that $m\in {\mathbb Z}_+$ and $p>1$.
Let $\a_p$ be an arbitrary Young function such that $\bar \a_p\in B_{p'}$. Next, let $\b_{p,m}$ be an arbitrary Young function such that
$\b_{p,m}^{-1}(t)\f^{-1}(t)\simeq\frac{t}{\log^m(e+t)}$, where $\f\in B_p$.
Then
$$
\|A_{L(\log L)^m,{\mathcal S}}\|_{L^p(v)\to L^p(u)}\lesssim [u^{1/p},v^{-1/p}]_{\a_p,\b_{p,m}}.
$$
If $T$ is a Calder\'on-Zygmund operator with Dini-continuous kernel, then
\begin{equation}\label{tbmes}
\|T_b^m\|_{L^p(v)\to L^p(u)}\lesssim \|b\|_{BMO}^m(
[u^{1/p},v^{-1/p}]_{\a_p,\b_{p,m}}+[u^{1/p},v^{-1/p}]_{\b_{p',m},\a_{p'}}).
\end{equation}
\end{theorem}

Observe that if $m=0$, then $\b_{p,0}=\a_{p'}$. Hence in this case the first and the second terms in (\ref{tbmes}) coincide, and we obtain the bump
conjecture for $T$ stated above.

In the case $m=1$, D. Cruz-Uribe and K. Moen \cite{CM} obtained different bump-type sufficient conditions. We give an example (in the Appendix) showing that the conditions in Theorem~\ref{extbctbm} provide
a wider class of weights $(u,v)$ for which $\|T_b^1\|_{L^p(v)\to L^p(u)}<\infty.$

The standard computations show that typical examples of $\a_p$ and $\b_{p,m}$ in Theorem~\ref{extbctbm} are
\begin{equation}\label{stco}
\a_p(t)=t^p\log^{p-1+\d}(e+t),\quad \b_{p,m}(t)=t^{p'}\log^{(m+1)p'-1+\d}(e+t).
\end{equation}

Our next result is an extension of the separated bump conjecture with logarithmic bumps to $T_b^m$. Here we fix $\a_p$ and $\b_{p,m}$ precisely as in (\ref{stco}).

\begin{theorem}\label{sepbumex} Let ${\mathcal S}$ be a sparse family. Assume that $m\in {\mathbb Z}_+$ and $p>1$. Take $\a_p$ and $\b_{p,m}$ as in (\ref{stco}). Define also
$$\ga_{p,m}(t)=t^{p'}\log^{m(p'+\d)}(e+t).$$
Then
$$
\|A_{L(\log L)^m,{\mathcal S}}\|_{L^p(v)\to L^p(u)}\lesssim [u^{1/p},v^{-1/p}]_{t^p,\b_{p,m}}+[u^{1/p},v^{-1/p}]_{\a_p,\ga_{p,m}}.
$$
If $T$ is a Calder\'on-Zygmund operator with Dini-continuous kernel, then
\begin{eqnarray}
\|T_b^m\|_{L^p(v)\to L^p(u)}&\lesssim& \|b\|_{BMO}^m\Big([u^{1/p},v^{-1/p}]_{t^p,\b_{p,m}}+[u^{1/p},v^{-1/p}]_{\a_p,\ga_{p,m}}\nonumber\\
&+& [u^{1/p},v^{-1/p}]_{\b_{p',m},t^{p'}}+[u^{1/p},v^{-1/p}]_{\ga_{p',m},\a_{p'}}\Big).\label{tbmnew}
\end{eqnarray}
\end{theorem}

The proof of this result is much more involved than the proof of Theorem \ref{extbctbm}. It is based on some extension of the works \cite{La,Li}.

If $m=0$, then $\b_{p,0}=\a_{p'}$ and $\ga_{p,0}=t^{p'}$. In this case the second line in (\ref{tbmnew}) coincides with the first line and we obtain
the above mentioned separated bump conjecture with logarithmic bumps from \cite{CRV}. In the case $m\ge 1$ one can see that our estimates are not totally in
the spirit of the separated bums. Simple manipulations with terms involving $\ga_{p,m}$ in (\ref{tbmnew}) allow us to get the following.

\begin{cor}\label{corpc} For $m\in {\mathbb N}$ define
$$\psi_{p,m}(t)=t^{p'}\log^{\max((m+1)p'-1, mp'+1)+\e}(e+t).$$ Then, with $T$ as above,
$$\|T_b^m\|_{L^p(v)\to L^p(u)}\lesssim\|b\|_{BMO}^m\Big([u^{1/p},v^{-1/p}]_{t^p,\psi_{p,m}}+[u^{1/p},v^{-1/p}]_{\psi_{p',m},t^{p'}}\Big).$$
\end{cor}

We conjecture that the terms with $\ga_{p,m}$ in (\ref{tbmnew}) can be fully avoided.

\begin{con}\label{consepb} Let $T$ be as above. Then
$$\|T_b^m\|_{L^p(v)\to L^p(u)}\lesssim\|b\|_{BMO}^m\Big([u^{1/p},v^{-1/p}]_{t^p,\b_{p,m}}+[u^{1/p},v^{-1/p}]_{\b_{p',m},t^{p'}}\Big).$$
\end{con}

It is interesting that Corollary \ref{corpc} coincides with Conjecture \ref{consepb} in the case $p=2$ but for every $p\not=2$, Conjecture \ref{consepb} provides a better result.

Turn to a necessary condition for the two-weighted boundedness of $T_b^m$.

\begin{theorem}\label{neccond}
Let $T$ be a non-degenerate Calder\'on-Zygmund operator. Let $m\in\mathbb{Z}_{+}$ and $p>1$. Assume that for every $b\in BMO$,
$$
\|T_{b}^{m}f\|_{L^{p,\infty}(u)}\lesssim\|b\|_{BMO}^m\|f\|_{L^{p}(v)}.
$$
Assume additionally that $u$ is a doubling weight. Then
\begin{equation}\label{necbump}
\sup_{Q}\|u^{1/p}\|_{L^p,Q}\|v^{-1/p}\|_{L^{p'}(\log L)^{mp'},Q}<\infty.
\end{equation}
\end{theorem}

Consider first the case $m=0$. Then (\ref{necbump}) is just the usual $A_p$ condition for $(u,v)$. In this case, when $T$ is the Hilbert transform, the necessity of the $A_p$ condition was obtained by
B. Muckenhoupt and R. Wheeden \cite{MW} without assuming the doubling condition on~$u$. However, it is not clear to us whether this condition can be removed, in general. In a very recent work \cite{CML},
a similar statement was obtained assuming a slightly weaker directionally doubling condition. In \cite{LSU1}, it was shown that the doubling condition can be avoided but assuming $T:L^p(v)\to L^{p,\infty}(u)$
for a family of operators.

To the best of our knowledge, in the case $m\ge 1$, Theorem \ref{neccond} is new. Its proof is based on a special construction of $BMO$ functions which goes back to P. Jones \cite{J}.

By duality, it follows from Theorem \ref{neccond} that if $T_b^{m+1}: L^p(v)\to L^p(u)$ and $u$ and $v^{1-p'}$ are doubling, then, additionally to (\ref{necbump}) (with $m+1$), we also have that
$$
\sup_{Q}\|u^{1/p}\|_{L^p(\log L)^{(m+1)p},Q}\|v^{-1/p}\|_{L^{p'},Q}<\infty.
$$
Hence, assuming also that $T$ is Dini-continuous, by Corollary \ref{corpc} for $m\ge 1$ and by the separated bump conjecture with logarithmic bumps for $m=0$, we obtain that
$T_b^{m}: L^p(v)\to L^p(u)$. Observe that by a well-known intuition, $T_b^1$ is more ``singular" than $T$. Therefore, it is natural to
expect that $T_b^{m+1}: L^p(v)\to L^p(u)$ implies $T_b^m: L^p(v)\to L^p(u)$ even without assuming the doubling conditions on $u$ and $v^{1-p'}$. However, we do not know any direct proof of this fact
even in the case $m=0$, and, in particular, it is not clear to us whether these doubling conditions can be removed. See also Remark \ref{noteond} for some comments about this.

In this paper we also obtain several results related to Bloom type estimates.
We recall the following result, see \cite{Hyt,LOR2} and the references therein. It's sufficiency part in the case $m=1$ is due to S. Bloom \cite{B}.

\begin{theorem}\label{bl} Let $T$ be a non-degenerate Calder\'on-Zygmund operator with Dini-continuous kernel.
Let $m\in\mathbb{N}$ and $p>1$. Assume that $\la,\mu\in A_p$. Further, let
$\eta=\big(\frac{\mu}{\la}\big)^{1/pm}$. Then
$$
b\in BMO_{\eta}\Rightarrow\|T_b^m\|_{L^p(\la)}\lesssim \|b\|_{BMO_{\eta}}^m\|f\|_{L^p(\mu)}
$$
and
$$
\|T_b^mf\|_{L^p(\la)}\lesssim \|f\|_{L^p(\mu)}\Rightarrow b\in BMO_{\eta}.
$$
\end{theorem}

It is natural to ask whether this result remains valid under more general assumptions on $\la,\mu$ and $\eta$.
We conjecture that this is not the case.

\begin{con}\label{conbl}
Let $m\in\mathbb{N}$ and $p>1$. Let $\la,\mu$ and $\eta$ be arbitrary weights such that the following holds:
$$
b\in BMO_{\eta}\Rightarrow\|T_b^m\|_{L^p(\la)}\lesssim \|b\|_{BMO_{\eta}}^m\|f\|_{L^p(\mu)}
$$
and
$$
\|T_b^mf\|_{L^p(\la)}\lesssim \|f\|_{L^p(\mu)}\Rightarrow b\in BMO_{\eta}.
$$
Then $\la,\mu\in A_p$ and $\eta\simeq \big(\frac{\mu}{\la}\big)^{1/pm}$.
\end{con}

Unluckily, our current tools are still far away from enabling us to establish this conjecture in full generality. However we still manage to provide several partial results in Section \ref{sec:Bloom}.

The remainder of the paper is organized as follows. Section \ref{sec:Prelim} contains some preliminary results and definitions. In Section \ref{sec:SparseB} we gather some sparse bounds that will be crucial to establish our main results. In Section \ref{sec:JoingBump} we prove Theorems~\ref{extbctbm} and~\ref{sepbumex}.
Section \ref{sec:necCond} is devoted to prove Theorem \ref{neccond}. In Section \ref{sec:Bloom} we discuss the partial results related to Conjecture \ref{conbl}.
Finally, in the Appendix we provide an example comparing our sufficient condition in Theorem \ref{extbctbm} with the condition from~\cite{CM}.

Throughout the paper we use the notation $A\lesssim B$ if $A\le CB$ with some independent constant $C$. We write $A\simeq B$ if $A\lesssim B$ and $B\lesssim A$.

\section{Preliminaries}\label{sec:Prelim}
\subsection{Calder\'on-Zygmund operators}
We say that $K$ is an $\omega$-Calder\'on-Zygmund kernel if $|K(x,y)|\le \frac{C_K}{|x-y|^n},x\not=y,$ and
$$|K(x,y)-K(x',y)|+|K(y,x)-K(y,x')|\le \o\left(\frac{|x-x'|}{|x-y|}\right)\frac{1}{|x-y|^n},$$
whenever $|x-y|>2|x-x'|$, where $\o:[0,1]\to [0,\infty)$ is the modulus of continuity.

A linear, $L^2$ bounded operator $T$ is called $\omega$-Calder\'on-Zygmund if it has a representation
$$Tf(x)=\int_{{\mathbb R}^n}K(x,y)f(y)dy\quad\text{for all}\,\,x\not\in \text{supp}\,f,$$
where $K$ is an $\omega$-Calder\'on-Zygmund kernel.

An $\omega$-Calder\'on-Zygmund kernel is called Dini-continuous if $\int_0^1\o(t)\frac{dt}{t}<\infty.$

We say that an $\omega$-Calder\'on-Zygmund kernel is non-degenerate if $\o(t)\to 0$ as $t\to 0$, and for every $y\in {\mathbb R}^n$ and $r>0$, there exists $x\not\in B(y,r)$ with
$$|K(x,y)|\ge \frac{c_0}{r^n}.$$
This definition was given by T. Hyt\"onen \cite{Hyt} (in the case when $K(x,y)=k(x-y)$, a similar notion was introduced by E. Stein \cite[p. 210]{S}).

We sat that $T$ is a non-degenerate Calder\'on-Zygmund operator if $T$ and its adjoint $T^*$ are associated with non-degenerate kernels. In other words,
we require that if $T$ is associated with kernel $K$, then $K(x,y)$ and $\tilde K(x,y)=K(y,x)$ are non-degenerate.

The following result is contained in the proof of \cite[Prop 2.2]{Hyt}.

\begin{prop}\label{h}
Let $K$ be a non-degenerate kernel. Then for every $A\ge 3$ and every ball $B(y_0, r)$,
there is a disjoint ball $\tilde B=B(x_0,r)$ at distance $\text{dist}(B,\tilde B)\simeq Ar$ such that
$$
|K(x_0,y_0)|\simeq \frac{1}{A^nr^n},
$$
and for all $y\in B$ and $x\in \tilde B$, we have
$$
|K(x,y)-K(x_0,y_0)|\lesssim \frac{\e_A}{A^nr^n},
$$
where $\e_A\to 0$ as $A\to \infty$.
\end{prop}

\subsection{Dyadic lattices and sparse families}
By a cube in ${\mathbb R}^n$ we mean a half-open cube $Q=\prod_{i=1}^n[a_i,a_i+h), h>0$. Denote by $\ell_Q$ the sidelength of $Q$.
Given a cube $Q_0\subset {\mathbb R}^n$, let ${\mathcal D}(Q_0)$ denote the set of all dyadic cubes with respect to $Q_0$, that is, the cubes
obtained by repeated subdivision of $Q_0$ and each of its descendants into $2^n$ congruent subcubes.

A dyadic lattice ${\mathscr D}$ in ${\mathbb R}^n$ is any collection of cubes such that
\begin{enumerate}
\renewcommand{\labelenumi}{(\roman{enumi})}
\item
if $Q\in{\mathscr D}$, then each child of $Q$ is in ${\mathscr D}$ as well;
\item
every 2 cubes $Q',Q''\in {\mathscr D}$ have a common ancestor, i.e., there exists $Q\in{\mathscr D}$ such that $Q',Q''\in {\mathcal D}(Q)$;
\item
for every compact set $K\subset {\mathbb R}^n$, there exists a cube $Q\in {\mathscr D}$ containing $K$.
\end{enumerate}
For this definition we refer to \cite{LN}.

A family of cubes ${\mathcal S}$ is called sparse if there exists $0<\a<1$ such that for every $Q\in {\mathcal S}$
one can find a measurable set $E_Q\subset Q$ with $|E_Q|\ge \a|Q|$, and the sets $\{E_Q\}$ are pairwise
disjoint.

In our results the sparseness constant $\a$ will depend only on $n$, and its precise value will not be relevant.

\subsection{Young functions and normalized Luxemburg norms}
By a Young function we mean a continuous,
convex, strictly increasing function $\f:[0,\infty)\to [0,\infty)$
with $\f(0)=0$ and $\f(t)/t\to \infty$ as $t\to \infty$.

The Orlicz maximal operator $M_{\f}$ is defined for a Young function $\f$ by
$$M_{\f}f(x)=\sup_{Q\ni x}\|f\|_{\f,Q},$$
where the supremum is taken over all cubes $Q$ containing the point $x$. If $\f(t)=t$, then $M_{\f}=M$
is the standard Hardy-Littlewood maximal operator.
It was shown by C. P\'erez \cite{P} that
\begin{equation}\label{p}
\f\in B_p\Rightarrow \|M_{\f}f\|_{L^p}\lesssim \|f\|_{L^p}\quad(p>1).
\end{equation}

Given a Young function $\f$, its complementary function is defined by
$$\bar\f(t)=\sup_{s\ge 0}\big(st-\f(s)\big).$$
Then $\bar\f$ is also a Young function satisfying $t\le \bar\f^{-1}(t)\f^{-1}(t)\le 2t$.
Also the following H\"older type estimate holds:
\begin{equation}\label{prop2}
\frac{1}{|Q|}\int_Q|f(x)g(x)|dx\le 2\|f\|_{\f,Q}\|g\|_{\bar\f,Q}.
\end{equation}

There is a more general version of (\ref{prop2}).
\begin{lemma}\label{holder} Let $A,B$ and $C$ be non-negative, continuous, strictly increasing functions on $[0,\infty)$
satisfying $A^{-1}(t)B^{-1}(t)\le C^{-1}(t)$ for all $t\ge 0$. Assume also that $C$ is convex. Then
$$
\|fg\|_{C,Q}\le 2\|f\|_{A,Q}\|g\|_{B,Q}.
$$
\end{lemma}

This lemma was proved by R. O'Neil \cite{ON} under the assumption that $A,B$ and $C$ are Young functions but the same proof works under the above conditions (see \cite{LOR1} for the details).

Recall the well-known facts (see, e.g., \cite[Ch. 10]{W}) that
\begin{equation}\label{maxlog}
\frac{1}{|Q|}\int_QM^k(f\chi_Q)\lesssim \|f\|_{L(\log L)^k,Q},
\end{equation}
where $M^k$ denotes the  $k$-th iteration of the maximal operator,
and
\begin{equation}\label{eqlog}
\|f\|_{L(\log L)^{\a},Q}\simeq \frac{1}{|Q|}\int_Q|f(x)|\log^{\a}\left(\frac{|f(x)|}{|f|_Q}+e\right)dx.
\end{equation}

Dealing with the logarithmic bumps of the form $t^p\log^{\a}(e+t)$, we will frequently use the following estimates. First, since $\||f|^r\|_{\f,Q}= \|f\|_{\f(t^r),Q}^r$ for any $r>0$, we have
$$\||f|^{1/p}\|_{L^p(\log L)^{\a}, Q}\simeq \|f\|_{L(\log L)^{\a},Q}^{1/p}.$$
Second, it follows from (\ref{eqlog}) and H\"older's inequality that for any $0<\d<1$,
$$\|f\|_{L(\log L)^{\a},Q}\lesssim (|f|_Q)^{\d}\|f\|_{L(\log L)^{\frac{\a}{1-\d}},Q}^{1-\d}.$$

\section{Sparse bounds for iterated commutators}\label{sec:SparseB}
Let $T$ be a Calder\'on-Zygmund operator with Dini-continuous kernel, and assume that $b$ is a locally integrable function.
Consider the $m$-th iterated commutator $T_b^m$.

Recall the following pointwise sparse bound established for $m=1$ in \cite{LOR1} and for $m\ge 1$ in \cite{IR}: for every bounded and compactly supported $f$,
there exist $3^n$ sparse families ${\mathcal S}_j\subset {\mathscr D_j}$ such that for a.e. $x\in {\mathbb R}^n$,
\begin{equation}\label{point}
|T_{b}^{m}f(x)|\lesssim \sum_{j=1}^{3^{n}}\sum_{Q\in\mathcal{S}_{j}}\sum_{k=0}^{m}|b(x)-b_{Q}|^{m-k}\left(\frac{1}{|Q|}\int_{Q}|b-b_{Q}|^{k}|f|\right)\chi_{Q}(x).
\end{equation}

Next, it was shown in \cite[Lemma 5.1]{LOR1} that given a sparse family ${\mathcal S}\subset {\mathscr D}$, there exists a sparse family $\tilde{\mathcal{S}}\subset {\mathscr D}$
containing $\mathcal{S}$ and such that if $Q\in\tilde{\mathcal{S}}$,
then for a.e. $x\in Q$,
\begin{equation}\label{sparses}
|b(x)-b_{Q}|\leq 2^{n+2}\sum_{P\in\tilde{\mathcal{S}},\ P\subseteq Q}\left(\frac{1}{|P|}\int_{P}|b-b_{P}|\right)\chi_{P}(x).
\end{equation}

Suppose now that $\eta$ is a weight, and $b\in BMO_{\eta}$, namely,
$$\|b\|_{BMO_{\eta}}=\sup_{Q}\frac{1}{\eta(Q)}\int_Q|b(x)-b_Q|dx<\infty.$$
Then (\ref{sparses}) allows to transform the right-hand side of (\ref{point}) in the following way.

Given a sparse family ${\mathcal S}$, define the sparse operator $A_{\mathcal S}$ by
$$A_{\mathcal S}f(x)=\sum_{Q\in {\mathcal S}}f_Q\chi_Q(x).$$
Further, consider the operator $A_{\mathcal S,\eta}$ defined by
$$A_{\mathcal S,\eta}f(x)=\eta A_{\mathcal S}f(x).$$
Denote by $A_{\mathcal S,\eta}^m$ the $m$-th iteration of $A_{\mathcal S,\eta}$.

\begin{lemma}\label{dual1} Let $m\in {\mathbb N}$. For every bounded and compactly supported $f$ and for every $b\in BMO_{\eta}$,
there exist $3^n$ sparse families ${\mathcal S}_j\subset {\mathscr D_j}$ such that for every $g\ge 0$,
\begin{equation}\label{fv}
|\langle T_b^mf,g\rangle|\lesssim \|b\|_{BMO_{\eta}}^m\sum_{j=1}^{3^n}\langle A_{{\mathcal S}_j}(A_{{\mathcal S}_j,\eta}^m|f|),g\rangle.
\end{equation}
\end{lemma}

This result is contained implicitly in \cite{LOR1} for $m=1$ and in \cite{LOR2} for $m\ge 1$. We outline briefly its proof for the sake of the completeness.

\begin{proof}[Proof of Lemma \ref{dual1}]
By (\ref{point}),
\begin{equation}\label{1dual}
|\langle T_b^mf,g\rangle|\lesssim \sum_{j=1}^{3^{n}}\sum_{Q\in\mathcal{S}_{j}}\sum_{k=0}^{m}
(|b-b_Q|^{k}|f|)_Q(|b-b_Q|^{m-k}|g|)_Q|Q|.
\end{equation}

Next, by (\ref{sparses}), there exist extended families $\tilde{\mathcal{S}_j}$ such that for a.e. $x\in Q$,
$$
|b(x)-b_{Q}|\lesssim \|b\|_{BMO_{\eta}}\sum_{P\in\tilde{\mathcal{S}_j},\ P\subseteq Q}\eta_P\chi_P(x).
$$
Since the cubes from $\tilde{\mathcal S}_j$ are dyadic, for every $l\in {\mathbb N}$,
$$
\Biggl(\sum_{P\in\tilde{\mathcal{S}_j},\ P\subseteq Q}\eta_P\chi_{P}\Biggr)^{l}
\lesssim \sum_{P_l\subseteq P_{l-1}\subseteq\dots\subseteq P_1\subseteq Q,\, P_i\in {\tilde{\mathcal S}_j}}\eta_{P_1}\eta_{P_2}\dots \eta_{P_l}\chi_{P_l}.
$$
Combining this estimate with the previous yields
$$
\int_{Q}|b-b_Q|^l|h|\lesssim \|b\|_{BMO_{\eta}}^l\int_{Q}A_{\tilde{\mathcal S_j},\eta}^l|h|.
$$

Therefore, replacing the right-hand side of (\ref{1dual}) by a larger sum over $\tilde{{\mathcal S}_j}$ and redenoting $\tilde{{\mathcal S}_j}$ again by ${\mathcal S}_j$, we obtain that
$$|\langle T_b^mf,g\rangle|\lesssim \|b\|_{BMO_{\eta}}^m\sum_{j=1}^{3^{n}}\sum_{Q\in\mathcal{S}_{j}}\sum_{k=0}^{m}
(A_{{\mathcal S_j},\eta}^{k}|f|)_Q(A_{{\mathcal S_j},\eta}^{m-k}g)_Q|Q|,$$
where $A_{{\mathcal S_j},\eta}^{0}|f|=|f|$. It remains to observe that, since $A_{\mathcal S}$ is self-adjoint,
\begin{eqnarray*}
&&\sum_{Q\in\mathcal{S}}(A_{{\mathcal S},\eta}^{k}|f|)_Q(A_{{\mathcal S},\eta}^{m-k}g)_Q|Q|=\langle A_{\mathcal S}(A_{{\mathcal S},\eta}^{k}|f|),A_{{\mathcal S},\eta}^{m-k}g\rangle\\
&&=\langle \eta A_{\mathcal S}(A_{{\mathcal S},\eta}^{k}|f|),A_{\mathcal S}(A_{{\mathcal S},\eta}^{m-k-1}g)\rangle=
\langle A_{\mathcal S}(A_{{\mathcal S},\eta}^{k+1}|f|),A_{{\mathcal S},\eta}^{m-k-1}g\rangle\\
&&=\dots=\langle A_{\mathcal S}(A_{{\mathcal S},\eta}^{m}|f|),g\rangle.
\end{eqnarray*}
This along with the previous estimate completes the proof.
\end{proof}

Lemma \ref{dual1} says that if $b\in BMO$ (that is, $\eta=1$), then $T_b^m$ is controlled by $A_{\mathcal S}^{m+1}$.
Let us show that in this case the right-hand side of (\ref{fv}) can be further transformed.

Given $m\in {\mathbb N}$, define the operator $T_{m}$ by
$$T_{m}f(x)=\sum_{Q\in {\mathcal S}}\left(\frac{1}{|Q|}\sum_{P_1\in {\mathcal S}:P_1\subseteq Q}\dots\sum_{P_m\in {\mathcal S}:P_m\subseteq P_{m-1}}\int_{P_m}f\right)\chi_{Q}(x),$$
where we assume that $P_0=Q$. Observe that the adjoint operator $T_m^*$ is given by
$$T_m^*f(x)=\sum_{Q\in {\mathcal S}}\left(\sum_{P_1\in {\mathcal S}:Q\subseteq P_1}\dots\sum_{P_m\in {\mathcal S}:P_{m-1}\subseteq P_{m}}f_{P_m}\right)\chi_{Q}(x)$$
(this can be easily checked by changing the summation and switching the indices).

\begin{lemma}\label{dual2} For all $f,g\ge 0$ and for every $m\in {\mathbb N}$,
$$
\langle A_{\mathcal S}^{m+1}f,g\rangle \lesssim \langle T_{m}f,g\rangle+\langle T_{m}^*f,g\rangle.
$$
\end{lemma}

\begin{proof} The proof is by induction on $m$. In the case $m=1$ we have
\begin{eqnarray*}
&&\langle A_{\mathcal S}^{2}f,g\rangle=\sum_{Q\in {\mathcal S}}\sum_{P\in {\mathcal S}}
\frac{|P\cap Q|}{|P|}g_{Q}\int_{P}f\\
&&=\sum_{Q\in {\mathcal S}}\sum_{P\subseteq Q}g_{Q}\int_{P}f+\sum_{Q\in {\mathcal S}}\sum_{P:Q\subset P}f_{P}\int_{Q}g\le \langle T_{1}f,g\rangle+\langle T_{1}^*f,g\rangle.
\end{eqnarray*}

Suppose that the lemma is true for $m=k-1$, and let us prove it for $m=k$, $k\ge 2$. Since $A_{\mathcal S}$ is self-adjoint and by the inductive assumption,
\begin{eqnarray*}
\langle A_{\mathcal S}^{k+1}f,g\rangle &=& \langle A_{\mathcal S}^{k}f,A_{\mathcal S}g\rangle\\
&\lesssim& \langle T_{k-1}f,A_{\mathcal S}g\rangle+\langle T_{k-1}^*f,A_{\mathcal S}g\rangle\\
&=& \langle f,T_{k-1}^*A_{\mathcal S}g\rangle+\langle f,T_{k-1}A_{\mathcal S}g\rangle.
\end{eqnarray*}

It follows from this that the proof will be completed if we show that
\begin{equation}\label{twopoint}
\max\big(T_{k-1}A_{\mathcal S}g, T_{k-1}^*A_{\mathcal S}g\big)\lesssim T_kg+T_k^*g.
\end{equation}

Consider first $T_{k-1}A_{\mathcal S}g$. We have
\begin{eqnarray}
\int_{P_{k-1}}A_{\mathcal S}g&=&\sum_{P_k\in {\mathcal S}}\frac{|P_k\cap P_{k-1}|}{|P_k|}\int_{P_k}g\label{comp}\\
&=&\sum_{P_k\in {\mathcal S}: P_k\subseteq P_{k-1}}\int_{P_k}g+\sum_{P_k\in {\mathcal S}: P_{k-1}\subset P_k}\frac{|P_{k-1}|}{|P_k|}\int_{P_k}g.\nonumber
\end{eqnarray}
Therefore,
\begin{equation}\label{fp}
T_{k-1}A_{\mathcal S}g\le T_kg+\sum_{Q\in {\mathcal S}}a_k(g,Q)\chi_{Q},
\end{equation}
where
\begin{equation}\label{cm}
a_k(g,Q)=\frac{1}{|Q|}\sum_{P_1\in {\mathcal S}:P_1\subseteq Q}\dots\sum_{P_{k-1}\in {\mathcal S}:P_{k-1}\subseteq P_{k-2}}
\sum_{P_k\in {\mathcal S}: P_{k-1}\subset P_k}\frac{|P_{k-1}|}{|P_k|}\int_{P_k}g.
\end{equation}

Transform the last two sums in (\ref{cm}) as follows:
\begin{eqnarray*}
\sum_{P_{k-1}\subseteq P_{k-2}}
\sum_{P_k: P_{k-1}\subset P_k}&=&\sum_{P_{k-1}\subseteq P_{k-2}}
\sum_{P_k: P_{k-1}\subset P_k\subseteq P_{k-2}}+
\sum_{P_{k-1}\subseteq P_{k-2}}
\sum_{P_k: P_{k-2}\subset P_k}\\
&=&\sum_{P_{k}\subseteq P_{k-2}}\sum_{P_{k-1}\subset P_k}+\sum_{P_{k-1}\subseteq P_{k-2}}
\sum_{P_k: P_{k-2}\subset P_k}.
\end{eqnarray*}

By the sparseness,
$$
\sum_{P_{k}\subseteq P_{k-2}}\sum_{P_{k-1}\subset P_k}
\frac{|P_{k-1}|}{|P_k|}\int_{P_k}g\lesssim \sum_{P_{k}\subseteq P_{k-2}}\int_{P_k}g
$$
and
$$
\sum_{P_{k-1}\subseteq P_{k-2}}
\sum_{P_k: P_{k-2}\subset P_k}\frac{|P_{k-1}|}{|P_k|}\int_{P_k}g\lesssim \sum_{P_k: P_{k-2}\subset P_k}\frac{|P_{k-2}|}{|P_k|}\int_{P_k}g.
$$

It follows from this that
\begin{equation}\label{iter}
\sum_{Q\in {\mathcal S}}a_k(g,Q)\chi_{Q}\lesssim T_{k-1}g+\sum_{Q\in {\mathcal S}}a_{k-1}(g,Q)\chi_{Q}
\end{equation}
and
$$\sum_{Q\in {\mathcal S}}a_{2}(g,Q)\chi_{Q}\lesssim T_1g+T_1^*g.$$
Therefore, iterating (\ref{iter}) yields
$$
\sum_{Q\in {\mathcal S}}a_k(g,Q)\chi_{Q}\lesssim\sum_{j=1}^{k-1}T_jg+T_1^*g\lesssim T_kg+T_k^*g,
$$
which, along with (\ref{fp}), proves the first part of (\ref{twopoint}).

The proof of the second part of (\ref{twopoint}) is similar. Consider $T_{k-1}^*A_{\mathcal S}g$.  Applying (\ref{comp}),
we obtain
\begin{equation}\label{fp1}
T_{k-1}^*A_{\mathcal S}g\le T_k^*g+\sum_{Q\in {\mathcal S}}b_k(g,Q)\chi_{Q},
\end{equation}
where
\begin{equation}\label{bm}
b_k(g,Q)=\frac{1}{|Q|}
\sum_{P_1\in {\mathcal S}:Q\subseteq P_1}\dots\sum_{P_{k-1}\in {\mathcal S}:P_{k-2}\subseteq P_{k-1}}
\sum_{P_k\in {\mathcal S}: P_k\subseteq P_{k-1}}\frac{1}{|P_{k-1}|}\int_{P_k}g
\end{equation}

Transform the last two sums in (\ref{bm}) as follows:
\begin{eqnarray*}
\sum_{P_{k-1}:P_{k-2}\subseteq P_{k-1}}
\sum_{P_k\subseteq P_{k-1}}&=&\sum_{P_{k-1}:P_{k-2}\subseteq P_{k-1}}\sum_{P_k\subseteq P_{k-2}}+
\sum_{P_{k-1}:P_{k-2}\subseteq P_{k-1}}\sum_{P_{k-2}\subset P_k\subseteq P_{k-1}}\\
&=&\sum_{P_{k-1}:P_{k-2}\subseteq P_{k-1}}\sum_{P_k\subseteq P_{k-2}}+
\sum_{P_{k}:P_{k-2}\subset P_{k}}\sum_{P_{k-1}:P_k\subseteq P_{k-1}}.
\end{eqnarray*}

Using the standard fact that $\sum_{Q\in {\mathscr D}:P\subseteq Q}\frac{1}{|Q|}\lesssim \frac{1}{|P|}$, we obtain that
$$
\sum_{P_{k-1}:P_{k-2}\subseteq P_{k-1}}\sum_{P_k\subseteq P_{k-2}}\frac{1}{|P_{k-1}|}\int_{P_k}g\lesssim
\sum_{P_k\subseteq P_{k-2}}
\frac{1}{|P_{k-2}|}\int_{P_k}g
$$
and
$$
\sum_{P_{k}:P_{k-2}\subset P_{k}}\sum_{P_{k-1}:P_k\subseteq P_{k-1}}\frac{1}{|P_{k-1}|}\int_{P_k}g
\lesssim \sum_{P_{k}:P_{k-2}\subset P_{k}}\frac{1}{|P_{k}|}\int_{P_k}g.
$$

It follows from this that
\begin{equation}\label{iter1}
\sum_{Q\in {\mathcal S}}b_k(g,Q)\chi_{Q}\lesssim T_{k-1}^*g+\sum_{Q\in {\mathcal S}}b_{k-1}(g,Q)\chi_{Q}
\end{equation}
and
$$\sum_{Q\in {\mathcal S}}b_{2}(g,Q)\chi_{Q}\lesssim T_1g+T_1^*g.$$
Therefore, iterating (\ref{iter1}) yields
$$
\sum_{Q\in {\mathcal S}}b_k(g,Q)\chi_{Q}\lesssim\sum_{j=1}^{k-1}T_j^*g+T_1g\lesssim T_kg+T_k^*g,
$$
which, along with (\ref{fp1}), proves the second part of (\ref{twopoint}).

This completes the proof of (\ref{twopoint}), and therefore, the lemma is proved.
\end{proof}

We end this section with a lemma which will be quite useful for our purposes.
\begin{lemma}\label{Lem:Am+1LlogLm}
Let $T$ be a Calder\'on-Zygmund operator with Dini-continuous kernel.
Let $m\in {\mathbb Z}_+$ and $p>1$. Let $u,v$ be weights. Then
$$
\|T_b^m\|_{L^p(v)\to L^p(u)}\lesssim
\|A_{L(\log L)^m,{\mathcal S}}\|_{L^p(v)\to L^p(u)}
+\|A_{L(\log L)^m,{\mathcal S}}\|_{L^{p'}(u^{1-p'})\to L^{p'}(v^{1-p'})}.
$$
\end{lemma}

\begin{proof}
First we note that by Lemmata \ref{dual1}, \ref{dual2} and by duality,
\begin{equation*}
\|T_b^m\|_{L^p(v)\to L^p(u)}\lesssim  \|T_m\|_{L^p(v)\to L^p(u)}
+\|T_m\|_{L^{p'}(u^{1-p'})\to L^{p'}(v^{1-p'})}.
\end{equation*}
Consider the expression defining $T_m$.
Since ${\mathcal S}$ is sparse,
$$\sum_{P_m\in {\mathcal S}: P_m\subseteq P_{m-1}}\int_{P_m}f\lesssim \sum_{P_m\in {\mathcal S}: P_m\subseteq P_{m-1}}\int_{E_{P_m}}M(f\chi_{P_0})\lesssim \int_{P_{m-1}}M(f\chi_{P_0}).$$
Applying subsequently this argument implies
$$\sum_{P_1\in {\mathcal S}: P_1\subseteq P_0}\dots\sum_{P_m\in {\mathcal S}: P_m\subseteq P_{m-1}}\int_{P_m}f \lesssim \int_{P_{0}}M^m(f\chi_{P_0}).$$
Therefore, by (\ref{maxlog}),
\begin{equation*}\label{logm}
T_{m}f(x)\lesssim \sum_{Q\in {\mathcal S}}\|f\|_{L(\log L)^m,Q}\chi_Q(x)
\end{equation*}
and we are done.
\end{proof}

\section{Proofs of Theorems \ref{extbctbm} and \ref{sepbumex}}\label{sec:JoingBump}
\subsection{Proof of Theorem \ref{extbctbm}}
By Lemma \ref{Lem:Am+1LlogLm}, it suffices to prove the first part of Theorem \ref{extbctbm} for $A_{L(\log L)^m,{\mathcal S}}$.
Consider the bilinear form
$$\langle A_{L(\log L)^m,{\mathcal S}}f, g\rangle=\sum_{Q\in S}\|f\|_{L(\log L)^m,Q}g_Q|Q|.$$

By H\"older's inequality (\ref{prop2}),
$$g_Q\le 2\|u^{1/p}\|_{\a_p,Q}\|gu^{-1/p}\|_{\bar\a_p,Q}.$$
Further, using the fact that if $C(t)=t\log^m(e+t)$, then $C^{-1}(t)\sim\frac{t}{\log^m(e+t)}$, and applying Lemma \ref{holder}, we obtain
$$\|f\|_{L(\log L)^m,Q}\lesssim   \|fv^{1/p}\|_{\f,Q}\|v^{-1/p}\|_{\b_{p,m},Q}.$$
Therefore,
\begin{equation}\label{thal}
\langle A_{L(\log L)^m,{\mathcal S}}f, g\rangle\lesssim [u^{1/p},v^{-1/p}]_{\a_p,\b_{p,m}}\sum_{Q\in S}\|fv^{1/p}\|_{\f,Q}\|gu^{-1/p}\|_{\bar\a_p,Q}|Q|.
\end{equation}

Using that ${\mathcal S}$ is sparse and by H\"older's inequality along with (\ref{p}),
\begin{eqnarray*}
&&\sum_{Q\in S}\|fv^{1/p}\|_{\f,Q}\|gu^{-1/p}\|_{\bar\a_p,Q}|Q|
\lesssim \sum_{Q\in {\mathcal S}}\int_{E_Q}M_{\f}(fv^{1/p})M_{\bar\a_p}(gu^{-1/p})dx\\
&&\lesssim \int_{{\mathbb R}^n}M_{\f}(fv^{1/p})M_{\bar\a_p}(gu^{-1/p})dx
\lesssim \|M_{\f}(fv^{1/p})\|_{L^p}\|M_{\bar\a_p}(gu^{-1/p})\|_{L^{p'}}\\
&&\lesssim \|f\|_{L^p(v)}\|g\|_{L^{p'}(u^{1-p'})}.
\end{eqnarray*}
This combined with (\ref{thal}) proves, by duality, the desired estimate for $A_{L(\log L)^m,{\mathcal S}}$, and therefore, Theorem \ref{extbctbm} is proved.

\vskip 0.3cm
The proof of Theorem \ref{sepbumex} requires some preliminaries which we mention in the following subsection.

\subsection{Auxiliary statements}
Given a sparse family ${\mathcal S}$ and a non-negative sequence
$\{\tau_Q\}_{Q\in {\mathcal S}}$, consider the operator
$T_{{\mathcal S},\tau}$ defined by
$$T_{{\mathcal S},\tau}f(x)=\sum_{Q\in {\mathcal S}}\tau_Qf_Q\chi_Q(x).$$
Given a cube $R$, denote ${\mathcal S}(R)=\{Q\in {\mathcal S}:Q\subseteq R\}$ and
$$
T_{{\mathcal S},\tau}^Rf(x)=\sum_{Q\in {\mathcal S}(R)}\tau_Qf_Q\chi_Q(x).
$$

The following result is due to M. Lacey, E. Sawyer and I. Uriarte-Tuero \cite{LSU} (see also \cite{H,T} for different proofs).

\begin{theorem}\label{lsu} Let $p>1$. We have
$$\|T_{{\mathcal S},\tau}(\cdot\si)\|_{L^p(\si)\to L^p(u)}\sim \sup_{R\in {\mathcal S}}\frac{\|T_{{\mathcal S},\tau}^R(\si)\|_{L^p(u)}}{\si(R)^{1/p}}
+\sup_{R\in {\mathcal S}}\frac{\|T_{{\mathcal S},\tau}^R(u)\|_{L^{p'}(\si)}}{u(R)^{1/p'}}.
$$
\end{theorem}

Let $p>1$. Suppose that $A\in B_p$ and $\f$ is a decreasing function such that $\int_{1/2}^{\infty}\frac{1}{\f(t)^{p'}}\frac{dt}{t}<\infty.$
In \cite{La}, M. Lacey established that the condition
\begin{equation}\label{lac}
\sup_Q(u_Q)^{1/p}\|\si^{1/p'}\|_{\bar A,Q}\f\left(\frac{\|\si^{1/p'}\|_{\bar A,Q}}{(\si_Q)^{1/p'}}\right)<\infty
\end{equation}
implies that
$$
\Big\|\sum_{Q\in {\mathcal S}(R)}\si_Q\chi_Q\Big\|_{L^p(u)}\lesssim \si(R)^{1/p}.
$$
It was also shown in \cite{La} that this result implies a particular case of the separated bump conjecture with logarithmic bumps proved in \cite{CRV}.

In \cite{Li}, K. Li provided a different proof of a slightly stronger result where (\ref{lac}) is replaced by
\begin{equation}\label{lii}
\sup_Q(u_Q)^{1/p}\frac{\si_Q}{\|\si^{1/p}\|_{A,Q}}
\f\left(\frac{(\si_Q)^{1/p}}{\|\si^{1/p}\|_{A,Q}}
\right)<\infty.
\end{equation}
Observe that, by H\"older's inequality,
$$
\frac{(\si_Q)^{1/p}}{\|\si^{1/p}\|_{A,Q}}\le 2\frac{\|\si^{1/p'}\|_{\bar A,Q}}{(\si_Q)^{1/p'}},
$$
and therefore (\ref{lac}) is stronger than (\ref{lii}).

We will need the following extension of the above results.

\begin{theorem}\label{estli} Let $p>1$, and let $\f$ and $\psi$ be increasing functions such that
$$\int_{1/2}^{\infty}\Big(\frac{1}{\f(t)^{p'}}+\frac{1}{\psi(t)^{p'}}\Big)\frac{dt}{t}<\infty.$$
Let ${\mathcal S}$ be a sparse family, and let $\{\la_Q\}_{Q\in {\mathcal S}}$ be a sequence such that $\la_Q\ge 1$ for every $Q\in {\mathcal S}$.
If $A\in B_p$ and
$$
K\equiv\sup_Q(u_Q)^{1/p}\la_Q\psi(\la_Q)\frac{\si_Q}{\|\si^{1/p}\|_{A,Q}}\f\left(\frac{(\si_Q)^{1/p}}{\|\si^{1/p}\|_{A,Q}}\right)<\infty,
$$
then
$$
\Big\|\sum_{Q\in {\mathcal S}(R)}\la_Q\si_Q\chi_Q\Big\|_{L^p(u)}\lesssim K\si(R)^{1/p}.
$$
\end{theorem}

The proof of this result is a minor modification of the corresponding proof in \cite{Li}. In particular, as in \cite{Li}, it is based on the two following statements.

\begin{prop}\label{cov}\cite[Proposition 2.2]{COV} let ${\mathscr D}$ be a dyadic lattice, and let $p>1$. For any non-negative sequence $\{a_Q\}_{Q\in {\mathscr D}}$ and for every weight $w$,
$$\Big\|\sum_{Q\in {\mathscr D}}a_Q\chi_Q\Big\|_{L^p(w)}\simeq \left(\sum_{Q\in {\mathscr D}}a_Q\Big(\frac{1}{w(Q)}\sum_{Q'\in {\mathscr D}, Q'\subseteq Q}a_{Q'}w(Q')\Big)^{p-1}w(Q)\right)^{1/p}.$$
\end{prop}

\begin{prop}\label{ah}\cite[Lemma 5.2]{H} Let ${\mathcal S}$ be a sparse family, and let $0<s<1$. For every weight $w$,
$$\sum_{Q\in {\mathcal S}, Q\subseteq R}(w_Q)^{s}|Q|\lesssim (w_R)^{s}|R|.$$
\end{prop}

\begin{proof}[Proof of Theorem \ref{estli}] For $k,m\ge 0$ define the sets
$${\mathcal S}_{k,m}=\Big\{Q\in {\mathcal S}(R): 2^k\le \la_Q\le 2^{k+1}, 2^m\le \frac{(\si_Q)^{1/p}}{\|\si^{1/p}\|_{A,Q}}\le 2^{m+1}\Big\}.$$
Then, applying Proposition \ref{cov} yields
\begin{eqnarray}
&&\Big\|\sum_{Q\in {\mathcal S}(R)}\la_Q\si_Q\chi_Q\Big\|_{L^p(u)}\le 2\sum_{k,m\ge 0}2^k\Big\|\sum_{Q\in {\mathcal S}_{k,m}}\si_Q\chi_Q\Big\|_{L^p(u)}\label{km}\\
&&\lesssim \sum_{k,m\ge 0}2^k\left(\sum_{Q\in {\mathcal S}_{k,m}}\si_Q\Big(\frac{1}{u(Q)}\sum_{Q'\in {\mathcal S}_{k,m}, Q'\subseteq Q}\si_{Q'}u(Q')\Big)^{p-1}u(Q)\right)^{1/p}.\nonumber
\end{eqnarray}

Suppose first that $p\ge 2$. Then, by Proposition \ref{ah},
\begin{eqnarray*}
\sum_{Q'\in {\mathcal S}_{k,m}, Q'\subseteq Q}\si_{Q'}u(Q')&=&\sum_{Q'\in {\mathcal S}_{k,m}, Q'\subseteq Q}\si_{Q'}(u_{Q'})^{\frac{1}{p-1}}(u_{Q'})^{1-\frac{1}{p-1}}|Q'|\\
&\lesssim& \left(\frac{K}{2^k\psi(2^k)2^m\f(2^m)}\right)^{p'}\sum_{Q'\in {\mathcal S}_{k,m}, Q'\subseteq Q}(u_{Q'})^{1-\frac{1}{p-1}}|Q'|\\
&\lesssim& \left(\frac{K}{2^k\psi(2^k)2^m\f(2^m)}\right)^{p'}(u_{Q})^{1-\frac{1}{p-1}}|Q|.
\end{eqnarray*}
Therefore, by (\ref{km}),
\begin{equation}\label{ar}
\Big\|\sum_{Q\in {\mathcal S}(R)}\la_Q\si_Q\chi_Q\Big\|_{L^p(u)}\lesssim K\sum_{k,m\ge 0}\frac{1}{\psi(2^k)2^m\f(2^m)}\Big(\sum_{Q\in {\mathcal S}_{k,m}}\si(Q)\Big)^{1/p}.
\end{equation}
From this, using that $\si(Q)\lesssim 2^{mp}\|\si^{1/p}\|_{A,Q}^p|Q|$ for $Q\in {\mathcal S}_{k,m}$, we obtain
\begin{eqnarray*}
&&\Big\|\sum_{Q\in {\mathcal S}(R)}\la_Q\si_Q\chi_Q\Big\|_{L^p(u)}\lesssim K\sum_{k,m\ge 0}\frac{1}{\psi(2^k)\f(2^m)}\Big(\sum_{Q\in {\mathcal S}_{k,m}}\|\si^{1/p}\|_{A,Q}^p|Q|\Big)^{1/p}\\
&&\lesssim K\left(\sum_{k,m\ge 0}\Big(\frac{1}{\psi(2^k)\f(2^m)}\Big)^{p'}\right)^{1/p'}\left(\sum_{k,m\ge 0}\sum_{Q\in {\mathcal S}_{k,m}}\|\si^{1/p}\|_{A,Q}^p|Q|\right)^{1/p}\\
&&\lesssim K\left(\int_{1/2}^{\infty}\frac{1}{\psi(t)^{p'}}dt\right)^{1/p'}\left(\int_{1/2}^{\infty}\frac{1}{\f(t)^{p'}}dt\right)^{1/p'}
\left(\sum_{Q\in {\mathcal S}(R)}\|\si^{1/p}\|_{A,Q}^p|Q|\right)^{1/p}\\
&&\lesssim \left(\int_RM_{A}(\si^{1/p}\chi_R)^p\right)^{1/p}\lesssim \si(R)^{1/p}.
\end{eqnarray*}

Consider now the case $1<p<2$. Then, by Proposition \ref{ah},
\begin{eqnarray*}
\sum_{Q'\in {\mathcal S}_{k,m}, Q'\subseteq Q}\si_{Q'}u(Q')&=&\sum_{Q'\in {\mathcal S}_{k,m}, Q'\subseteq Q}(\si_{Q'})^{p-1}u_{Q'}(\si_{Q'})^{2-p}|Q'|\\
&\lesssim& \left(\frac{K}{2^k\psi(2^k)2^m\f(2^m)}\right)^{p}\sum_{Q'\in {\mathcal S}_{k,m}, Q'\subseteq Q}(\si_{Q'})^{2-p}|Q'|\\
&\lesssim& \left(\frac{K}{2^k\psi(2^k)2^m\f(2^m)}\right)^{p}(\si_{Q})^{2-p}|Q|.\\
\end{eqnarray*}
Therefore, by (\ref{km}),
\begin{eqnarray*}
&&\Big\|\sum_{Q\in {\mathcal S}(R)}\la_Q\si_Q\chi_Q\Big\|_{L^p(u)}\\
&&\lesssim \sum_{k,m\ge 0}2^k\left(\frac{K}{2^k\psi(2^k)2^m\f(2^m)}\right)^{p-1}
\left(\sum_{Q\in {\mathcal S}_{k,m}}\si_Q\big((\si_Q)^{p-1}u_Q\big)^{2-p}|Q|\right)^{1/p}\\
&&\lesssim K\sum_{k,m\ge 0}\frac{1}{\psi(2^k)2^m\f(2^m)}\Big(\sum_{Q\in {\mathcal S}_{k,m}}\si(Q)\Big)^{1/p},
\end{eqnarray*}
and we again arrived at (\ref{ar}), which completes the proof.
\end{proof}

\subsection{Proof of Theorem \ref{sepbumex}}\label{subsec:Proofmcrv} As before,
by Lemma \ref{Lem:Am+1LlogLm}, it suffices to establish the first part of Theorem \ref{sepbumex} for $A_{L(\log L)^m,{\mathcal S}}$.

Note that in \cite{Li}, K. Li found a characterization of a similar inequality
$$\|\sum_{Q\in {\mathcal S}}\|f\|_{L^r,Q}\chi_Q\|_{L^p(u)}\lesssim \|f\|_{L^p(v)} \quad(1<r<p).$$
We partially follow his approach.

It will be more convenient to deal with an equivalent form of the statement written as
\begin{equation}\label{eqf}
\|\sum_{Q\in {\mathcal S}}\|f\si\|_{L(\log L)^m,Q}\chi_Q\|_{L^p(u)}\lesssim \|f\|_{L^p(\si)},
\end{equation}
where $\si=v^{1-p'}$. Note that in terms of $\si$, the assumptions that
$$[u^{1/p},v^{-1/p}]_{t^p,\b_{p,m}}+[u^{1/p},v^{-1/p}]_{\a_p,\ga_{p,m}}<\infty$$
can be rewritten in the form
\begin{equation}\label{recond}
\sup_Qu_Q\|\si\|_{L(\log L)^{(m+1)p'-1+\d},Q}^{p-1}<\infty
\end{equation}
and
\begin{equation}\label{k2}
\sup_Q\|u\|_{L(\log L)^{p-1+\d},Q}\|\si\|_{L(\log L)^{m(p'+\d)},Q}^{p-1}<\infty.
\end{equation}

We will use the notation
$$\|f\|_{\f,Q}^{\si}=\inf\left\{\la>0:\frac{1}{\si(Q)}\int_Q\f(|f(y)|/\la)\si(y)dy\right\}$$
and $f_{Q,\si}=\frac{1}{\si(Q)}\int_Qf\si$.

We start by observing that
\begin{eqnarray}
&&\|f\si\|_{L(\log L)^m,Q}\simeq\frac{1}{|Q|}\int_Qf\log^m\left(\frac{f\si}{(f\si)_Q}+e\right)\si\label{si}\\
&&=\frac{1}{|Q|}\int_Qf\log^m\left(\frac{f}{f_{Q,\si}}\frac{\si}{\si_Q}+e\right)\si\nonumber\\
&&\lesssim \frac{1}{|Q|}\int_Qf\log^m\left(\frac{f}{f_{Q,\si}}+e\right)\si+\frac{1}{|Q|}\int_Qf\log^m\left(\frac{\si}{\si_Q}+e\right)\si.\nonumber
\end{eqnarray}

Take $1<r<p$, which will be specified later on.
By H\"older's inequality,
\begin{eqnarray*}
\frac{1}{|Q|}\int_Qf\log^m\left(\frac{\si}{\si_Q}+e\right)\si&\le& \si_Q\|f\|_{L^r,Q}^\si\left(\frac{1}{\si(Q)}\int_Q\log^{mr'}\left(\frac{\si}{\si_Q}+e\right)\si\right)^{1/r'}\\
&\simeq& \|\si\|_{L(\log L)^{mr'},Q}^{1/r'}(\si_Q)^{1/r}\|f\|_{L^r,Q}^\si.
\end{eqnarray*}
Next,
$$\frac{1}{|Q|}\int_Qf\log^m\left(\frac{f}{f_{Q,\si}}+e\right)\si\simeq \si_Q\|f\|_{L(\log L)^m,Q}^{\si}\lesssim \|\si\|_{L(\log L)^{mr'},Q}^{1/r'}(\si_Q)^{1/r}\|f\|_{L^r,Q}^\si.$$
Therefore, by (\ref{si}),
$$
\|f\si\|_{L(\log L)^m,Q}\lesssim \|\si\|_{L(\log L)^{mr'},Q}^{1/r'}(\si_Q)^{1/r}\|f\|_{L^r,Q}^\si.
$$

We obtain that (\ref{eqf}) will follow from
\begin{equation}\label{sec}
\|\sum_{Q\in {\mathcal S}}\|\si\|_{L(\log L)^{mr'},Q}^{1/r'}(\si_Q)^{1/r}\|f\|_{L^r,Q}^\si\chi_Q\|_{L^p(u)}\lesssim \|f\|_{L^p(\si)}.
\end{equation}
Observe that (\ref{sec}) is equivalent to
\begin{equation}\label{seceq}
\|\sum_{Q\in {\mathcal S}}\|\si\|_{L(\log L)^{mr'},Q}^{1/r'}(\si_Q)^{1/r}\Big(\frac{1}{\si(Q)}\int_Qf\si\Big)\chi_Q\|_{L^p(u)}\lesssim \|f\|_{L^p(\si)}.
\end{equation}
Indeed, on the one hand, (\ref{seceq}) follows from (\ref{sec}) by H\"older's inequality. On the other hand, since
$$\|f\|_{L^r,Q}^\si\le \frac{1}{\si(Q)}\int_Q(M_{r,\si}^{\mathscr D}f)\si,$$
and $M_{r,\si}^{\mathscr D}$ is bounded on $L^p(\si)$ (here is important that $r<p$), we obtain that
(\ref{seceq}) implies (\ref{sec}).

Denote
$$\la_Q=\left(\frac{\|\si\|_{L(\log L)^{mr'}, Q}}{\si_Q}\right)^{1/r'}.$$
By Theorems \ref{lsu} and \ref{estli}, in order to establish (\ref{seceq}), it suffices to show that there exist
$A\in B_p, B\in B_{p'}$ and functions $\f,\psi, \rho,\theta$ satisfying
$$\int_{1/2}^{\infty}\Big(\frac{1}{\f(t)^{p'}}+\frac{1}{\psi(t)^{p'}}\Big)\frac{dt}{t}<\infty\quad\text{and}\quad
\int_{1/2}^{\infty}\Big(\frac{1}{\rho(t)^{p}}+\frac{1}{\theta(t)^{p}}\Big)\frac{dt}{t}<\infty
$$
such that
\begin{equation}\label{cht1}
\sup_Q(u_Q)^{1/p}\la_Q\psi(\la_Q)
\|\si^{1/p'}\|_{\bar A,Q}\f
\left(\frac{\|\si^{1/p'}\|_{\bar A,Q}}{(\si_Q)^{1/p'}}\right)<\infty
\end{equation}
and
\begin{equation}\label{cht2}
\sup_Q(\si_Q)^{1/p'}\la_Q\rho(\la_Q)
\|u^{1/p}\|_{\bar B,Q}\theta
\left(\frac{\|u^{1/p}\|_{\bar B,Q}}{(u_Q)^{1/p}}\right)<\infty.
\end{equation}

We start by verifying (\ref{cht1}). In what follows we introduce several parameters that will be fixed later on.
Take $\f(t)=\psi(t)=\log(e+t)$. Next, let $A(t)=\frac{t^p}{\log^{1+\mu}(e+t)}$. Then
$A\in B_p$ and $\bar A(t)\sim t^{p'}\log^{\frac{1+\mu}{p-1}}(e+t)$.

Take $0<\nu<1$ such that $\frac{1+\mu}{1-\nu}=1+2\mu$. Then, by H\"older's inequality,
$$
\|\si^{1/p'}\|_{\bar A,Q}\sim \|\si\|_{L(\log L)^{\frac{1+\mu}{p-1}},Q}^{1/p'}
\lesssim \left((\si_Q)^{\nu}
\|\si\|_{L(\log L)^{\frac{1+2\mu}{p-1}},Q}^{1-\nu}\right)^{1/p'}.
$$
Hence, setting
$$
t_Q=\frac{\|\si\|_{L(\log L)^{\frac{1+2\mu}{p-1}},Q}}{\si_Q}
$$
and using that $\sup_{t\ge 1}t^{-\frac{\nu}{p'}}
\f\big(t^{\frac{1-\nu}{p'}}\big)
<\infty$,
we obtain
\begin{eqnarray}
\|\si^{1/p'}\|_{\bar A,Q}\f
\left(\frac{\|\si^{1/p'}\|_{\bar A,Q}}{(\si_Q)^{1/p'}}\right)&\lesssim&
\|\si\|_{L(\log L)^{\frac{1+2\mu}{p-1}},Q}^{1/p'}
t_Q^{-\frac{\nu}{p'}}\f\big(t_Q^{\frac{1-\nu}{p'}}\big)\label{sipart}\\
&\lesssim& \|\si\|_{L(\log L)^{\frac{1+2\mu}{p-1}},Q}^{1/p'}.\nonumber
\end{eqnarray}

Similarly, by H\"older's inequality, if $s<r$, then
$$\left(\frac{\|\si\|_{L(\log L)^{mr'}, Q}}{\si_Q}\right)^{1/r'}\lesssim \left(\frac{\|\si\|_{L(\log L)^{ms'}, Q}}{\si_Q}\right)^{1/s'}.$$
Therefore, setting
$$\tau_Q=\frac{\|\si\|_{L(\log L)^{ms'},Q}}{\si_Q}$$
and using that $\sup_{t\ge 1}t^{\frac{1}{s'}-\frac{1}{r'}}\psi(t^{1/s'})<\infty$, we obtain
\begin{eqnarray}
\la_Q\psi(\la_Q)&\lesssim& \tau_Q^{1/s'}\psi(\tau_Q^{1/s'})=\tau_Q^{1/r'}\tau_Q^{1/s'-1/r'}\psi(\tau_Q^{1/s'})\label{laq}\\
&\lesssim & \left(\frac{\|\si\|_{L(\log L)^{ms'}, Q}}{\si_Q}\right)^{1/r'}.\nonumber
\end{eqnarray}
From this, and by (\ref{sipart}), the left-hand side of (\ref{cht1}) is controlled by
\begin{equation}\label{contby}
(u_Q)^{1/p} \|\si\|_{L(\log L)^{\frac{1+2\mu}{p-1}},Q}^{1/p'}\left(\frac{\|\si\|_{L(\log L)^{ms'}, Q}}{\si_Q}\right)^{1/r'}.
\end{equation}

Let $0<q<1$ and $s_0<s$. Then, by H\"older's inequality, the expression in (\ref{contby}) is at most
\begin{equation}\label{newcont}
(u_Q)^{1/p}\left((\si_Q)^q\|\si\|_{L(\log L)^{\frac{1+2\mu}{(p-1)(1-q)}},Q}^{1-q}\right)^{1/p'}
\left(\frac{\|\si\|_{L(\log L)^{ms_0'}, Q}}{\si_Q}\right)^{s'/s_0'r'}.
\end{equation}

We now fix the parameters in such a way that
$$\frac{q}{p'}=\frac{s'}{s_0'r'}\quad\text{and}\quad \frac{1}{(p-1)(1-q)}=\frac{ms_0'r'}{s'}.$$
It follows from this that
$$q=\frac{mp}{mp+1}\quad\text{and}\quad ms_0'=\frac{s'}{r'}\big((m+1)p'-1\big).$$
Since $s_0<s$, we have $s_0'>s'$, and hence $r'<\frac{(m+1)p'-1}{m}$. Therefore, the additional assumption on $r$ is that $\frac{mp+1}{m+1}<r$.

Let $\d$ be a constant from condition (\ref{recond}). Take $s<r$ in such a way that
\begin{equation}\label{sr}
\frac{s'}{r'}\big((m+1)p'-1\big)\le(m+1)p'-1+\d.
\end{equation}
Also fix $\mu>0$ such that $\frac{2\mu}{(p-1)(1-q)}=\d$. We obtain that the expression in (\ref{newcont}) is at most
$$(u_Q)^{1/p}\|\si\|_{L(\log L)^{(m+1)p'-1+\d},Q}^{1/p'},$$
which, by (\ref{recond}), proves (\ref{cht1}).

The proof of (\ref{cht2}) is based on similar ideas.
As before, set $\rho(t)=\theta(t)=\log(e+t)$. Next, let $B(t)=\frac{t^{p'}}{\log^{1+\mu}(e+t)}$. Then
$B\in B_{p'}$ and $\bar B(t)\sim t^{p}\log^{(1+\mu)(p-1)}(e+t)$.

The same arguments as above show that
$$
\|u^{1/p}\|_{\bar B,Q}\theta
\left(\frac{\|u^{1/p}\|_{\bar B,Q}}{(u_Q)^{1/p}}\right)\lesssim \|u\|_{L(\log L)^{(1+2\mu)(p-1)},Q}^{1/p}.
$$
From this and from (\ref{laq}) we obtain that the left-hand side of (\ref{cht2}) is at most
\begin{eqnarray}
&&\|u\|_{L(\log L)^{(1+2\mu)(p-1)},Q}^{1/p}(\si_Q)^{1/p'}
 \left(\frac{\|\si\|_{L(\log L)^{ms'}, Q}}{\si_Q}\right)^{1/r'}\nonumber\\
 &&\lesssim
 \|u\|_{L(\log L)^{(1+2\mu)(p-1)},Q}^{1/p}
 \|\si\|_{L(\log L)^{ms'}, Q}^{1/p'}.\label{ull}
 \end{eqnarray}

Observe that our current assumptions on $s$ and $r$ (guaranteeing that (\ref{cht1}) holds) are $\frac{mp+1}{m+1}<r<p$ and $s<r$ such that (\ref{sr}) holds.
We now assume additionally that $s$ and $r$ are so close to $p$ that $s'\le p'+\d$. Fix also $\mu$ such that $2\mu(p-1)=\d$.
Then we obtain that the expression in (\ref{ull}) is controlled by condition (\ref{k2}), and therefore, Theorem \ref{sepbumex} is proved.

\begin{proof}[Proof of Corollary \ref{corpc}]
Recall that
$$
\b_{p,m}(t)=t^{p'}\log^{(m+1)p'-1+\d}(e+t),\quad \ga_{p,m}(t)=t^{p'}\log^{m(p'+\d)}(e+t)
$$
and
$$
\psi_{p,m}(t)=t^{p'}\log^{\max((m+1)p'-1, mp'+1)+\e}(e+t).
$$
It suffices to prove that, with suitable choice of $\d$,
$$
[u^{1/p},v^{-1/p}]_{t^p,\b_{p,m}}+[u^{1/p},v^{-1/p}]_{\a_p,\ga_{p,m}}
\lesssim
[u^{1/p},v^{-1/p}]_{t^p,\psi_{p,m}}+[u^{1/p},v^{-1/p}]_{\psi_{p',m},t^{p'}}.
$$
This would provide the estimate for $\|A_{L(\log L)^m,{\mathcal S}}\|_{L^p(v)\to L^p(u)}$. Since the right-hand side here is self-dual,
from this and from Lemma \ref{Lem:Am+1LlogLm} we obtain the desired bound for $T_b^m$.

Observe that for $\d\le \e$ we have $\b_{p,m}\le \psi_{p,m}$. Therefore,
$$[u^{1/p},v^{-1/p}]_{t^p,\b_{p,m}}\le [u^{1/p},v^{-1/p}]_{t^p,\psi_{p,m}}.$$
Hence, the result will follow if we show that
\begin{equation}\label{itrm}
[u^{1/p},v^{-1/p}]_{\a_p,\ga_{p,m}}\lesssim [u^{1/p},v^{-1/p}]_{t^p,\psi_{p,m}}+[u^{1/p},v^{-1/p}]_{\psi_{p',m},t^{p'}}.
\end{equation}

By H\"older's inequality, for $0<\a<1$,
$$
\|u^{1/p}\|_{L^p(\log L)^{p-1+\d},Q}\lesssim \|u^{1/p}\|_{L^p,Q}^{\a}\|u^{1/p}\|_{L^p(\log L)^{\frac{p-1+\d}{1-\a}},Q}^{1-\a}
$$
and
$$
\|v^{-1/p}\|_{L^{p'}(\log L)^{m(p'+\d)}, Q}\lesssim \|v^{-1/p}\|_{L^{p'},Q}^{1-\a}\|v^{-1/p}\|_{L^{p'}(\log L)^{\frac{m(p'+\d)}{\a}},Q}^{\a}.
$$
Therefore, the left-hand side of (\ref{itrm}) is at most
\begin{equation}\label{bexpr}
\Big(\|u^{1/p}\|_{L^p,Q}\|v^{-1/p}\|_{L^{p'}(\log L)^{\frac{m(p'+\d)}{\a}},Q}\Big)^{\a}
\Big(\|v^{-1/p}\|_{L^{p'},Q}\|u^{1/p}\|_{L^p(\log L)^{\frac{p-1+\d}{1-\a}},Q}\Big)^{1-\a}.
\end{equation}

Fix $\a$ in such a way that $\frac{p-1}{1-\a}=(m+1)p-1$. Then $\a=\frac{mp}{mp+p-1}$ and $\frac{mp'}{\a}=mp'+1$.
Hence, taking $\d$ such that $\frac{\d m}{\a}\le \e$ and $\frac{\d}{1-\a}\le \e$, we obtain that the expression in (\ref{bexpr}) is
bounded by the right-hand side of (\ref{itrm}), and therefore, the proof is complete.
\end{proof}

\section{A necessary condition}\label{sec:necCond}
\subsection{On a theorem of P. Jones}
In \cite{J}, P. Jones established a rather general result allowing to decide whether a function from $BMO(\O)$, where $\O\subset {\mathbb R}^n$ is
a connected open set, can be extended to a function from $BMO({\mathbb R}^n)$.

We will need a particular version of this result when $\O=Q$ is a cube. Observe that the proof of a general result in \cite{J} is long and involved.
In the particular case we need, it is much simpler. Therefore we outline the proof below.

\begin{theorem}\label{bmoj} Assume that $f\in BMO$, and let $R$ be a cube such that $f_R=0$. Then there exists a function $\f$ such that $\f=f$ on $R$, $\f=0$ on ${\mathbb R}^n\setminus 2R$
and
$$
\|\f\|_{BMO}\lesssim \|f\|_{BMO}.
$$
\end{theorem}

\begin{remark}\label{gr}
In the one-dimensional case this statement can be found in \cite[Ex. 3.1.10, p. 167]{G2}.
\end{remark}

\begin{remark} The proof we give is an adaptation of the method in \cite{J}. In particular, as in \cite{J}, we shall make use of the Whitney covering theorem.
We refer to \cite[p. 348]{BS} for the statement and various properties of Whitney's cubes.
\end{remark}

\begin{proof}[Proof of Theorem \ref{bmoj}]
Let $E=\{Q_j\}$ and $E'=\{Q_j'\}$ be the Whitney coverings of the interiors of ${\mathbb R}^n\setminus R$ and $R$, respectively.

Take $\a=\a_n$ with the following property: for every $Q_j\in E$ with $\ell_{Q_j}\le \a\ell_R$ we have $Q_j\subset 2R$, and, moreover, there exists the
nearest cube $P_j'\in E'$ such that $|P_j'|\ge |Q_j|$. Denote
$$F=\{Q_j\in E: \ell_{Q_j}\le \a\ell_R\},$$
and define
$$\f=f\chi_R+\sum_{Q_j\in F}f_{P_j'}\chi_{Q_j}.$$
Observe that each cube $P_j'\in E'$ may appear in this sum not more than $k=k_n$ times.

Denote $\tilde R=R\cup(\cup_{Q_j\in F}Q_j)$. To prove that $\|\f\|_{BMO}\lesssim \|f\|_{BMO}$, it suffices to show that for every cube~$Q$, there exists $c\in {\mathbb R}$
such that
\begin{equation}\label{osc}
\frac{1}{|Q|}\int_{Q\cap R}|f-c|+\sum_{Q_j\in F}\frac{|Q\cap Q_j|}{|Q||P'_j|}\Big|\int_{P'_j}(f-c)\Big|+\frac{|Q\setminus \tilde R|}{|Q|}|c|
\lesssim \|f\|_{BMO}.
\end{equation}

Denote $A=\{Q_j\in F: Q\cap Q_j\not=\emptyset\}$. If $A=\emptyset$, then either $Q\subset R$ or $Q\subset {\mathbb R}^n\setminus {\tilde R}$,
and this case is trivial. Therefore, suppose that $A\not=\emptyset$.  There are two main cases.

\begin{enumerate}
\renewcommand{\labelenumi}{(\roman{enumi})}
\item Suppose that $\ell_{Q_j}\le 4\ell_{Q}$ for every $Q_j\in A$.

If $|Q\setminus \tilde R|>0$, then
there exists $Q_j\in A$ such that $\text{dist}(Q_j, \partial R)\sim \text{diam} R$, and hence
$|R|\lesssim |Q|$. In this case we take $c=f_R=0$.  Then the left-hand side of (\ref{osc}) is bounded by
$$\frac{1}{|R|}\int_{R}|f-f_R|+\frac{1}{|R|}\sum_{Q_j\in F}\int_{P'_j}|f-f_R|\lesssim \frac{1}{|R|}\int_{R}|f-f_R|\lesssim \|f\|_{BMO}.$$

Suppose that $|Q\setminus \tilde R|=0$. It follows from the definition of $P_j'$ that for every $Q_j\in A$ with $\ell_{Q_j}\le 4\ell_{Q}$ the corresponding
$P_j'$ is contained in $\b Q$, where $\b=\b_n$. Therefore, taking $c=f_{\b Q}$, we obtain that the left-hand side of (\ref{osc}) is bounded by
$$\frac{1}{|Q|}\int_{\b Q}|f-f_{\b Q}|+\frac{1}{|Q|}\sum_{Q_j\in F}\int_{P'_j}|f-f_{\b Q}|\lesssim \frac{1}{|Q|}\int_{\b Q}|f-f_{\b Q}|\lesssim \|f\|_{BMO}.$$

\item Suppose that there exists $Q_{j_0}\in A$ such that $\ell_Q<\frac{1}{4}\ell_{Q_{j_0}}$. Then $Q\subset \frac{3}{2}Q_{j_0}$, and hence $Q\cap R=\emptyset$. It follows from the properties of Whitney cubes that every other cube $Q_j\in A$ touches $Q_{j_0}$, and therefore $|Q_j|\sim |Q_{j_0}|$,
and the corresponding cube $P_j'$ is contained in $\ga Q_{j_0}$, where $\ga=\ga_n$.

Now, if $|Q\setminus \tilde R|=0$, we take $c=f_{\ga Q_{j_0}}$. Then the left-hand side of (\ref{osc}) is bounded by
$$\frac{1}{|Q_{j_0}|}\sum_{Q_j\in F}\int_{P'_j}|f-f_{\ga Q_{j_0}}|\lesssim
\frac{1}{|Q_{j_0}|}\int_{\ga Q_{j_0}}|f-f_{\ga Q_{j_0}}|\lesssim \|f\|_{BMO}.
$$

If $|Q\setminus \tilde R|\not=0$, then $|Q_{j_0}|\sim |R|$. In this case, taking $c=f_R$, we obtain that the left-hand side of (\ref{osc}) is bounded by
$$
\frac{1}{|Q_{j_0}|}\int_{\ga Q_{j_0}}|f-f_{R}|\lesssim \frac{1}{|R|}\int_{2\ga R}|f-f_{2\ga R}|\lesssim \|f\|_{BMO}.
$$
\end{enumerate}

This completes the proof of (\ref{osc}), and therefore, the theorem is proved.
\end{proof}

\subsection{Proof of Theorem \ref{neccond}} We start by observing that, by duality, the estimates
$$
\|T_{b}^{m}f\|_{L^{p,\infty}(u)}\lesssim\|b\|_{BMO}\|f\|_{L^{p}(v)}
$$
and
\begin{equation}\label{eqform}
\|(T_b^m)^*f\|_{L^{p'}(v^{1-p'})}\lesssim \|b\|_{BMO}\|f/u\|_{L^{p',1}(u)}
\end{equation}
are equivalent.

Note that
$$T_b^mf(x)=\int_{{\mathbb R}^n}\big(b(x)-b(y)\big)^mK(x,y)f(y)dy\quad(x\not\in\text{supp}\,f).$$
Hence $(T_b^m)^*$ is essentially the same operator but associated with $\tilde K(x,y)=K(y,x)$.
Since $\tilde K$ is non-degenerate (by our definition of a non-degenerate Calder\'on-Zygmund operator),
it suffices to prove the theorem assuming that (\ref{eqform}) holds for $T_b^m$ instead of $(T_b^m)^*$.

Next, by the reasons explained in Section 2.3, (\ref{necbump}) is equivalent to
\begin{equation}\label{bump}
\sup_{Q}\left(\frac{1}{|Q|}\int_Qu\right)\left(\frac{1}{|Q|}\int_Qv^{1-p'}\log^{mp'}\Big(\frac{v^{1-p'}}{(v^{1-p'})_Q}+e\Big)\right)^{p-1}<\infty.
\end{equation}

Let $Q$ be an arbitrary cube. Define
$$g(x)=\log^+\left(\frac{M(v^{1-p'}\chi_Q)(x)}{(v^{1-p'})_Q}\right).$$
It is well known \cite{CR} that $g\in BMO$ and $\|g\|_{BMO}\lesssim 1$.
Also, using that (see, e.g., \cite[Ex. 2.1.5, p. 100]{G1})
$$\int_Q(M(f\chi_Q))^{\d}\lesssim \left(\frac{1}{|Q|}\int_Q|f|\right)^{\d}|Q|\quad(0<\d<1),$$
we obtain
\begin{equation}\label{aver}
g_Q\lesssim 1.
\end{equation}

By Theorem \ref{bmoj}, there exists a function $\f$ such that $\f=g-g_Q$ on $Q$ and $\f=0$ outside $2Q$, and
$\|\f\|_{BMO}\lesssim 1$.
Let $B$ be the ball concentric with $Q$ or radius $r=\text{diam}\,Q$.
In accordance with Proposition \ref{h}, take the corresponding ball $\tilde B$ of the same radius
at distance $\text{dist}(B,\tilde B)\simeq Ar$, where $A\ge 3$ will be chosen later.

Let $f$ be a non-negative function supported in $\tilde B$. Set $b=\f$. Observe that $b$ is supported in $B$,
and hence $b=0$ on $\tilde B$. Thus, for $x\in B$,
\begin{eqnarray*}
T_b^mf(x)&=&\int_{\tilde B}(b(x)-b(y))^mK(x,y)f(y)dy\\
&=&\f^m(x)\int_{\tilde B}K(x,y)f(y)dy.
\end{eqnarray*}
Therefore, by (\ref{eqform}) (with $T_b^m$),
$$
\int_{B}\Big|\int_{\tilde B}K(x,y)f(y)dy\Big|^{p'}|\f|^{mp'}v^{1-p'}\, dx\lesssim \|f\chi_{\tilde B}/u\|_{L^{p',1}(u)}^{p'}.
$$

From this, and by Proposition \ref{h},
\begin{eqnarray*}
&&\frac{1}{A^n} \Big(\int_B|\f|^{mp'}v^{1-p'}\Big)^{1/p'}f_{\tilde B}\lesssim\left(\int_{B}\Big|\int_{\tilde B}K(x_0,y_0)f(y)dy\Big|^{p'}|\f|^{mp'}v^{1-p'}\,dx\right)^{1/p'}\\
&&\le \left(\int_{B}\Big(\int_{\tilde B}|K(x,y)-K(x_0,y_0)|f(y)dy\Big)^{p'}|\f|^{mp'}v^{1-p'}\,dx\right)^{1/p'}\\
&&+\left(\int_{B}\Big|\int_{\tilde B}K(x,y)f(y)dy\Big|^{p'}|\f|^{mp'}v^{1-p'}\, dx\right)^{1/p'}\\
&&\lesssim \frac{\e_A}{A^n}\Big(\int_B|\f|^{mp'}v^{1-p'}\Big)^{1/p'}f_{\tilde B}+\|f\chi_{\tilde B}/u\|_{L^{p',1}(u)}.
\end{eqnarray*}
Therefore, taking $A$ large enough, we obtain
\begin{equation}\label{fla}
\Big(\int_B|\f|^{mp'}v^{1-p'}\Big)^{1/p'}f_{\tilde B}\lesssim \|f\chi_{\tilde B}/u\|_{L^{p',1}(u)}.
\end{equation}

Setting here $f=u$ and using that $\|\chi_{\tilde B}\|_{L^{p',1}(u)}\simeq (\int_{\tilde B}u)^{1/p'}$ yields
$$\left(\frac{1}{|\tilde B|}\int_{\tilde B}u\right)
\left(\int_B|\f|^{mp'}v^{1-p'}\right)^{p-1}\lesssim 1.$$

Note that $|\tilde B|\simeq |Q|$ and $Q\subset \ga\tilde B$, where $\ga$ depends only on $A$ and $n$. Therefore, since $u$ is
doubling (recall that this means that there exists $c>0$ such that $u(2Q)\le cu(Q)$ for every cube $Q$), we obtain
\begin{equation}\label{gq}
\left(\frac{1}{|Q|}\int_{Q}u\right)
\left(\int_Q|g-g_Q|^{mp'}v^{1-p'}\right)^{p-1}\lesssim 1.
\end{equation}

We now observe that the same proof with the choice $b=\chi_{B}$ shows that (\ref{gq}) holds with $m=0$. Combining this with (\ref{aver})
yields
$$\left(\frac{1}{|Q|}\int_{Q}u\right)
\left(\int_Qg^{mp'}v^{1-p'}\right)^{p-1}\lesssim 1,$$
which, in turn, implies (\ref{bump}), and therefore, the theorem is proved.

\begin{remark}\label{noteond} As we have mentioned in the Introduction, if $T$ is non-degenerate Calder\'on-Zygmund operator with Dini-continuous kernel,
and if both weights $u$ and $v^{1-p'}$ are doubling, then $T_{b}^{m+1}:L^p(v)\to L^p(u)$ implies $T_b^{m}: L^p(v)\to L^p(u)$. We add two remarks here. First, it is easy to see that under the above assumptions
we actually obtain that $A_{\mathcal S}^{m+1}: L^p(v)\to L^p(u)$, and therefore in the conclusion $T_b^{m}: L^p(v)\to L^p(u)$, $T$ can be replaced by any Calder\'on-Zygmund operator with Dini-continuous kernel.

Second, the assumptions that both $u$ and $v^{1-p'}$ are doubling can be replaced by that either $u$ or $v^{1-p'}$ belongs to $A_{\infty}$. Indeed, assume, for example, that $u\in A_{\infty}$ and that $T_{b}^{m+1}:L^p(v)\to L^p(u)$. Then, by Theorem \ref{neccond},
$$\sup_{Q}\|u^{1/p}\|_{L^p,Q}\|v^{-1/p}\|_{L^{p'}(\log L)^{(m+1)p'},Q}<\infty.$$
Since $u\in A_{\infty}$, it satisfies the reverse H\"older inequality, and therefore, the above condition can be self-improved to
$$\sup_{Q}\|u^{1/p}\|_{L^{rp},Q}\|v^{-1/p}\|_{L^{p'}(\log L)^{(m+1)p'},Q}<\infty$$
with some $r>1$.
It is easy to see that this condition is stronger that the assumptions of Theorem \ref{sepbumex}, and therefore $T_b^{m}: L^p(v)\to L^p(u)$.
\end{remark}

\section{On a converse to Bloom's theorem} \label{sec:Bloom}
Throughout this section we assume that $T$ is a non-degenerate Calder\'on-Zygmund operator with Dini-continuous kernel, and $m\in {\mathbb N}$.

As we announced in the introduction, we obtain several partial results related to Conjecture \ref{conbl}, being the first of them the following theorem.

\begin{theorem}\label{fpc} Let $\la,\mu\in A_p, p>1$. Let $\eta$ be an arbitrary weight such that
\begin{equation}\label{cond1con}
b\in BMO_{\eta}\Rightarrow\|T_b^m\|_{L^p(\la)}\lesssim \|b\|_{BMO_{\eta}}^m\|f\|_{L^p(\mu)}
\end{equation}
and
\begin{equation}\label{cond2con}
\|T_b^mf\|_{L^p(\la)}\lesssim \|f\|_{L^p(\mu)}\Rightarrow b\in BMO_{\eta}.
\end{equation}
Then $\eta\simeq \big(\frac{\mu}{\la}\big)^{1/pm}$ almost everywhere.
\end{theorem}

In the following theorem we assume that $\eta=1$.

\begin{theorem}\label{spc} Let $p>1$. Let $\la$ and $\mu$ be the weights satisfying either one of the following conditions:
\begin{enumerate}
\renewcommand{\labelenumi}{(\roman{enumi})}
\item $\la\in A_p$ and $\mu$ is an arbitrary weight;
\item $\la\in A_{\infty}$ and $\mu^{1-p'}\in A_{\infty}$.
\end{enumerate}
Suppose also that
\begin{equation}\label{cond1un}
b\in BMO \Rightarrow\|T_b^m\|_{L^p(\la)}\lesssim \|b\|_{BMO}^m\|f\|_{L^p(\mu)}
\end{equation}
and
\begin{equation}\label{cond2un}
\|T_b^mf\|_{L^p(\la)}\lesssim \|f\|_{L^p(\mu)}\Rightarrow b\in BMO.
\end{equation}
Then $\la\simeq\mu$ almost everywhere and $\la,\mu\in A_p$.
\end{theorem}

As we will see below, Theorem \ref{spc} is more difficult than Theorem \ref{fpc}. In particular, in the simplest case when $\la=1$, $\mu$ is an arbitrary weight and $m=1$ this result says that
the implication $b\in BMO\Leftrightarrow \|T_b^1f\|_{L^p}\lesssim \|f\|_{L^p(\mu)}$ holds if and only if $\mu\sim 1$. Even in such a simple form this result seems to be new.

\subsection{Auxiliary propositions}
We first recall several standard properties of $A_p$ weights (see, e.g., \cite[Ch. 9]{G1}):
\begin{enumerate}
\renewcommand{\labelenumi}{(\roman{enumi})}
\item
if $w\in A_p$, then $w$ is a doubling weight;
\item if $w\in A_p$, then $w\in A_{\infty}$, which means that there exist $c,\rho>0$ such that for every cube $Q$ and any subset $E\subset Q$,
\begin{equation}\label{ainf}
w(E)\le c\left(\frac{|E|}{|Q|}\right)^{\rho}w(Q);
\end{equation}
\item if $w\in A_p$, then for every cube $Q$ and any subset $E\subset Q$,
\begin{equation}\label{conv}
|E|\le [w]_{A_p}^{1/p}\left(\frac{w(E)}{w(Q)}\right)^{1/p}|Q|,
\end{equation}
where
$$[w]_{A_p}=\sup_Q\left(\frac{1}{|Q|}\int_Qw\right)\left(\frac{1}{|Q|}\int_Qw^{-\frac{1}{p-1}}\right)^{p-1};$$
\item if $\la\in A_p$, then there exists $\e>0$ such that $\la^{1+\e}\in A_p$.
\end{enumerate}

\begin{prop}\label{neccon}
Let $\la$ and $\mu$ be arbitrary weights such that (\ref{cond1con}) holds. Then for every ball $B=B(y_0,r)$,
there is a disjoint ball $\tilde B=B(x_0,r)$ at distance $\text{dist}(B,\tilde B)\simeq r$ such that for any $f\ge 0$
\begin{equation}\label{appr1}
(f_B)^p(\eta^{pm}\la)(\tilde B)\lesssim \int_Bf^p\mu.
\end{equation}
\end{prop}

\begin{proof}
Let $B$ be an arbitrary ball, and let $\tilde B$ be the corresponding ball from Proposition \ref{h}.
Set $b=\eta\chi_B$. We trivially have that $\|b\|_{BMO_{\eta}}\le 2$. Then, following exactly the same
argument leaded to (\ref{fla}), we obtain (\ref{appr1}).
\end{proof}

\begin{cor}\label{pointmu}
Let $\la$ and $\mu$ be arbitrary weights such that (\ref{cond1con}) holds. Then
$$\la\eta^{pm}\lesssim \mu$$
almost everywhere.
\end{cor}

\begin{proof}
This follows immediately by taking $f=1$ in (\ref{appr1}) and applying the Lebesgue differentiation theorem.
\end{proof}

\begin{cor}\label{apc} Assume that $\la$ and $\mu$ satisfy the conditions of Theorem \ref{spc}. Then $(\la,\mu)\in A_p$.
\end{cor}

\begin{proof}
This follows at once from (\ref{appr1}) (with $\eta=1$) by taking $f=\mu^{1-p'}$ and using that $\la$ is doubling.
\end{proof}

\begin{prop}\label{unb}
Let $v$ be a weight such that $v\not\in L^{\infty}$. Then there exists a sequence $\a_k\uparrow \infty$ with the following property:
for every sequence $0<\d_k<1$ there exist pairwise disjoint cubes $Q_k$ and subsets $E_k\subset Q_k$ with $|E_k|\ge \d_k|Q_k|$ such that $\a_k<v\le 2\a_k$ on $E_k$.
\end{prop}

\begin{proof}
For $j\in {\mathbb N}$ denote
$$\O_j=\{x:2^j<v\le 2^{j+1}\}.$$
Since $v$ is unbounded, there is a subsequence $j_k\to \infty$ for which the sets $\O_{j_k}$ have positive measure. Denote $\a_k=2^{j_k}$.

By the standard density argument, there exist cubes $Q_k$ such that $|Q_k\cap \O_{j_k}|\ge \d_k|Q_k|$. Moreover, since the sets $\O_{j_k}$ are pairwise
disjoint, the cubes $Q_k$ can be taken pairwise disjoint as well (taking them as small as necessary). Setting $E_k=Q_k\cap \O_{j_k}$ completes the proof.
\end{proof}

\begin{lemma}\label{bmo}
Let $\eta_1,\eta_2$ be the weights such that $\frac{\eta_1}{\eta_2}\not\in L^{\infty}$. Then there exists
$b\in BMO_{\eta_1}\setminus BMO_{\eta_2}$.
\end{lemma}

\begin{proof} Apply Proposition \ref{unb} to the weight $v=\frac{\eta_1}{\eta_2}$ with $\d_k=\frac{1}{2}$.
We obtain that there exists an unbounded sequence $\alpha_{k}$
and there are sequences of pairwise disjoint cubes $Q_{k}$ and measurable
subsets $E_{k}\subset Q_{k}$ such that $|E_{k}|\ge\frac{1}{2}|Q_{k}|$
and $\a_k<\frac{\eta_1}{\eta_2}\le 2\a_k$ on~$E_k$.

Next, there exists $i_k\in {\mathbb Z}$ such that the set
$$E_k'=\{x\in E_k:2^{i_k}<\eta_2(x)\le 2^{i_k+1}\}$$
has positive measure. Denote $\b_k=2^{i_k}$.

Let $x_k$ be the density point of $E_k'$ and the Lebesque point of $\eta_2$. Then there exists a cube $R_k$ centered at $x_k$ and containing in $Q_k$ such that
$\eta_2(R_k\cap E_k')\ge \frac{1}{2}\eta_2(R_k)$ and $|R_k\cap E_k'|\ge \frac{1}{2}|R_k|$.

Choose $A_{k}\subset R_k\cap E_k'$ such that $|A_k|=\frac{1}{2}|R_k|$, and set
$$b=\sum_{k}\alpha_{k}\b_k\chi_{A_{k}}.$$
Since $b\le \eta_1$, we trivially obtain that $b\in BMO_{\eta_1}$.

Let us show that $b\not\in BMO_{\eta_2}$. Note that $b_{R_k}=\frac{1}{2}\a_k\b_k$. Hence
\begin{eqnarray*}
\frac{1}{\eta_2(R_k)}\int_{R_k}|b(x)-b_{R_k}|dx&\ge&\frac{1}{\eta_2(R_k)}\int_{A_k}|\a_k\b_k-\frac{1}{2}\a_k\b_k|dx\\
&=&\frac{1}{4}\a_k\b_k\frac{|R_k|}{\eta_2(R_k)}.
\end{eqnarray*}
On the other hand,
$$\eta_2(R_k)\le 2\eta_2(R_k\cap E_k')\le 4\b_k|R_k|.$$
Therefore,
$$
\frac{1}{\eta_2(R_k)}\int_{R_k}|b(x)-b_{R_k}|dx\ge \frac{1}{16}\a_k.
$$
This completes the proof since $\a_k$ is unbounded.
\end{proof}

\begin{lemma}\label{ap} Let $\la\in A_p$. Assume that $\mu$ is a weight such that $\la\le\mu$ and $\frac{\mu}{\la}\not\in L^{\infty}$. Then there exists
a weight $u\not\in L^{\infty}$ such that $(\la u, \mu)\in A_p$.
\end{lemma}

\begin{proof}
Apply Proposition \ref{unb} to the weight $v=\frac{\mu}{\la}$. We obtain the corresponding sequences $\a_k, \d_k, Q_k$ and $E_k$.
Denote  $G_k=Q_k\setminus E_k$. Then $|G_k|\le (1-\d_k)|Q_k|$. Set also $\si_{\la}=\la^{-\frac{1}{p-1}}$.

Let us show that taking suitable sequence $\d_k$ one can choose the sets $A_k\subset \frac{1}{2}Q_k$ of positive measure and satisfying the following properties:
\begin{enumerate}
\renewcommand{\labelenumi}{(\roman{enumi})}
\item $|A_k|=\ga_k|Q_k|$, where $\sum_{k}\a_k\ga_k^{\rho}<\infty,$ and
$\rho$ is the constant from the $A_{\infty}$ property (\ref{ainf}) with $w=\la$;
\item if $Q\subset Q_k$ and $Q\cap A_k\not=\emptyset$, then
\begin{equation}\label{prsi}
\si_{\la}(Q\cap G_k)\le\a_k^{-\frac{1}{p-1}}\si_{\la}(Q).
\end{equation}
\end{enumerate}

Define the weighted local maximal operator
$$M_{\si_{\la},Q_k}f(x)=\sup_{Q\ni x, Q\subset Q_k}\frac{1}{\si_{\la}(Q)}\int_Q|f|\si_{\la},$$
and consider the sets
$$B_k=\{x\in Q_k:M_{\si_{\la},Q_k}\chi_{G_k}(x)>\a_k^{-\frac{1}{p-1}}\}.$$

Observe that $\si_{\la}\in A_{p'}$, and therefore it is a doubling weight. Thus, by the weighted weak type estimate for $M_{\si_{\la}}$ along with (\ref{ainf}),
$$\si_{\la}(B_k)\le c_{n,\la}\a_k^{\frac{1}{p-1}}\si_{\la}(G_k)\le c_{n,\la}'\a_k^{\frac{1}{p-1}}(1-\d_k)^{\e}\si_{\la}(Q_k).$$
Hence, by (\ref{conv}),
\begin{eqnarray*}
|B_k|&\le& [\si_{\la}]_{A_{p'}}^{1/p'}\Big(c_{n,\la}'\a_k^{\frac{1}{p-1}}(1-\d_k)^{\e}\Big)^{1/p'}|Q_k|\\
&=&(c_{n,\la}')^{1/p'}(\a_k[w]_{A_p})^{1/p}(1-\d_k)^{\e/p'}|Q_k|.
\end{eqnarray*}
Take now $\d_k$ such that $(c_{n,\la}')^{1/p'}(\a_k[w]_{A_p})^{1/p}(1-\d_k)^{\e/p'}=\frac{1}{2^{n+1}}$.

We have that $|B_k|\le \frac{1}{2^{n+1}}|Q_k|$. Take $\ga_k$ such that $\sum_{k}\a_k\ga_k^{\rho}<~\infty$ and $\ga_k<\frac{1}{2^{n+1}}$ for all $k$.
Then there exists $A_k\subset \frac{1}{2}Q_k$ such that $|A_k|=\ga_k|Q_k|$ and $A_k\cap B_k=\emptyset$. We have that property (i) is satisfied, and
property~(ii) holds as well: if $Q\subset Q_k$ and $Q\cap A_k\not=\emptyset$, then $Q\cap (Q_k\setminus B_k)\not=\emptyset$, and hence, by the definition of $B_k$,
(\ref{prsi}) holds.

Taking the sets $A_k$ that satisfy properties (i) and (ii), define
$$u=\sum_k\a_k\chi_{A_k}+\chi_{{\mathbb R}^n\setminus \cup_kA_k}.$$
Let us show that $(\la u, \mu)\in A_p$.

Denote
$$F(Q)=\left(\frac{1}{|Q|}\int_Q\la u\right)\left(\frac{1}{|Q|}\int_Q\mu^{-\frac{1}{p-1}}\right)^{p-1}.$$

Assume that $Q$ is not contained in any $Q_k$ and $Q\cap \frac{1}{2}Q_k=\emptyset$ for every~$k$. Then
$$F(Q)=\left(\frac{1}{|Q|}\int_Q\la\right)\left(\frac{1}{|Q|}\int_Q\mu^{-\frac{1}{p-1}}\right)^{p-1}\le [\la]_{A_p}.$$

Assume that $Q$ is not contained in any $Q_k$ and the set
$${\mathcal K}=\{k:Q\not\subset Q_k, Q\cap \frac{1}{2}Q_k\not=\emptyset\}$$
is not empty. Then $Q_k\subset 7Q$ for every $k\in {\mathcal K}$. We obtain
\begin{equation}\label{st}
\int_Q\la u\le \sum_{k\in {\mathcal K}}\a_k\int_{A_k}\la+\int_Q\la.
\end{equation}
Applying (\ref{ainf}) along with the doubling property of $\la$ and (i) yields
\begin{eqnarray*}
&&\sum_{k\in {\mathcal K}}\a_k\int_{A_k}\la \le c\sum_{k\in {\mathcal K}}\a_k\ga_k^{\rho}\int_{Q_k}\la\\
&&\le c\left(\sum_{k}\a_k\ga_k^{\rho}\right)\int_{7Q}\la\le c'\int_Q\la.
\end{eqnarray*}
Combining this with (\ref{st}), we obtain $\int_Q\la u\le c\int_Q\la$, which implies
$$F(Q)\le c[\la]_{A_p}.$$

It remains to consider the case when there exists $k$ such that $Q\subset Q_k$.
Observe that
\begin{equation}\label{ob1}
\frac{1}{|Q|}\int_Q\la u\le \frac{1}{|Q|}\int_Q\la+\left(\frac{1}{|Q|}\int_{Q\cap A_k}\la\right)\a_k
\end{equation}
and
\begin{equation}\label{ob2}
\frac{1}{|Q|}\int_Q\mu^{-\frac{1}{p-1}}\le\a_k^{-\frac{1}{p-1}}\frac{1}{|Q|}\int_Q\la^{-\frac{1}{p-1}}+\frac{1}{|Q|}\int_{Q\cap G_k}\la^{-\frac{1}{p-1}}.
\end{equation}

If
$$\left(\frac{1}{|Q|}\int_{Q\cap A_k}\la\right)\a_k\le \frac{1}{|Q|}\int_Q\la,$$
then, by (\ref{ob1}),
$$F(Q)\le 2[\la]_{A_p}.$$

Assume that
$$\frac{1}{|Q|}\int_Q\la<\left(\frac{1}{|Q|}\int_{Q\cap A_k}\la\right)\a_k.$$
Then $Q\cap A_k\not=\emptyset$. Hence, by property (ii) along with (\ref{ob1}) and (\ref{ob2}),
$$\frac{1}{|Q|}\int_Q\la u\le 2\a_k\frac{1}{|Q|}\int_Q\la \quad\text{and}\quad \frac{1}{|Q|}\int_Q\mu^{-\frac{1}{p-1}}\le2\a_k^{-\frac{1}{p-1}}\frac{1}{|Q|}\int_Q\la^{-\frac{1}{p-1}}.$$
From this,
$$F(Q)\le 2^p[\la]_{A_p},$$
and therefore, the proof is complete.
\end{proof}

Property (iv) of $A_p$ weights mentioned in the beginning of this section along with Lemma \ref{ap} implies the following.

\begin{cor}\label{c}
Let $\la\in A_p$. Assume that $\mu$ is a weight such that $\la\le\mu$ and $\frac{\mu}{\la}\not\in L^{\infty}$. Then there exist $\e>0$ and a weight $u\not\in L^{\infty}$ such that $((\la u)^{1+\e}, \mu^{1+\e})\in A_p$.
\end{cor}

Finally, an important role in our proofs will be played by the following result of C. Neugebauer \cite{N}.

\begin{theorem}\label{N} Let $\la$ and $\mu$ be the weights such that $(\la^r,\mu^r)\in A_p, p>1,$ for some $r>0$. Then there exists a weight $w\in A_p$
such that
$$\la\lesssim w\lesssim \mu$$
almost everywhere.
\end{theorem}

\subsection{Proofs of Theorems \ref{fpc} and \ref{spc}}
\begin{proof}[Proof of Theorem \ref{fpc}]
By Corollary \ref{pointmu}, it remains to prove that
$$\mu\lesssim \la\eta^{pm}$$
almost everywhere. Assume that this is not true. Define $\tilde\eta=(\mu/\la)^{1/mp}$. Then $\tilde \eta/\eta\not\in L^{\infty}$.

In order to get a contradiction, it suffices to show that
\begin{equation}\label{contr}
\|A_{\mathcal S}(A^m_{\mathcal S,\tilde \eta}f)\|_{L^p(\la)}\lesssim \|f\|_{L^p(\mu)},
\end{equation}
where $A_{\mathcal S,\tilde\eta}^m$ is the $m$-th iteration of $A_{\mathcal S,\tilde\eta}f=\tilde\eta A_{\mathcal S}f$.
Indeed, from this, by Lemma~\ref{dual1}, we obtain that for $b\in BMO_{\tilde\eta}$,
$$\|T_b^mf\|_{L^p(\la)}\lesssim \|b\|_{BMO_{\tilde\eta}}^m\|f\|_{L^p(\mu)},$$
which, by Lemma \ref{bmo}, contradicts (\ref{cond2con}).

To show (\ref{contr}), we will use the well-known fact that $A_S$ is bounded on $L^p(w)$ for $w\in A_p$ (see, e.g.,\cite{CMP}).
Also, by H\"older's inequality,
$$\la\tilde\eta^{kp}=\la^{1-\frac{k}{m}}\mu^{\frac{k}{m}}\in A_p \quad(k=0,\dots,m).$$
Hence,
\begin{eqnarray*}
\|A_{\mathcal S}(A^m_{\mathcal S,\tilde \eta}f)\|_{L^p(\la)}&\lesssim& \|A^m_{\mathcal S,\tilde \eta}f\|_{L^p(\la)}\lesssim \|A^{m-1}_{\mathcal S,\tilde \eta}f\|_{L^p(\la\tilde\eta^p)}\\
&\lesssim &\dots\lesssim\|A_{\mathcal S}f\|_{L^p(\mu)}\lesssim \|f\|_{L^p(\mu)},
\end{eqnarray*}
proving (\ref{contr}).
\end{proof}

The proof of Theorem \ref{spc} is similar but now Lemma \ref{ap} along with Theorem~\ref{N} will play the crucial role.

\begin{proof}[Proof of Theorem \ref{spc}]
By Corollary \ref{pointmu}, $\la\lesssim \mu$ a.e., and therefore it remains to prove the converse estimate. Assume that this is not true. As in the previous proof, it suffices to show that
there exists a weight $u\not\in L^{\infty}$ such that
\begin{equation}\label{ansuf}
\|A_{\mathcal S}(A^m_{\mathcal S,u}f)\|_{L^p(\la)}\lesssim \|f\|_{L^p(\mu)}.
\end{equation}

Assume first that $\la\in A_p$ and $\mu$ is an arbitrary weight. Corollary \ref{c} along with Theorem \ref{N} shows that there exist $u\not\in L^{\infty}$ and $w\in A_p$ such that
$$\la u^{mp}\lesssim w\lesssim \mu.$$
It follows from this that
$$\la u^p\lesssim \la^{1-\frac{1}{m}}w^{\frac{1}{m}}.$$
Also, by H\"older's inequality, $\la^{1-\frac{1}{m}}w^{\frac{1}{m}}\in A_p$. Hence,
\begin{eqnarray*}
\|A_{\mathcal S}(A^m_{\mathcal S,u}f)\|_{L^p(\la)}&\lesssim& \|A^m_{\mathcal S,u}f\|_{L^p(\la)}=\|A_{\mathcal S}(uA^{m-1}_{\mathcal S,u}f)\|_{L^p(\la u^p)}\\
&\lesssim &\|A_{\mathcal S}(uA^{m-1}_{\mathcal S,u}f)\|_{L^p(\la^{1-\frac{1}{m}}w^{\frac{1}{m}})}\lesssim \|A_{\mathcal S, u}^{m-1}f\|_{L^p(\la^{1-\frac{1}{m}}w^{\frac{1}{m}}u^p)}.
\end{eqnarray*}

Arguing similarly, we have that
$$\la^{1-\frac{k-1}{m}}w^{\frac{k-1}{m}}u^p\lesssim \la^{1-\frac{k}{m}}w^{\frac{k}{m}}$$
and $\la^{1-\frac{k}{m}}w^{\frac{k}{m}}\in A_p$ for all $2\le k\le m$.
Therefore, iterating this argument yields
\begin{eqnarray*}
&&\|A_{\mathcal S, u}^{m-1}f\|_{L^p(\la^{1-\frac{1}{m}}w^{\frac{1}{m}}u^p)}\lesssim \|A_{\mathcal S}(uA^{m-2}_{\mathcal S,u}f)\|_{L^p(\la^{1-\frac{2}{m}}w^{\frac{2}{m}})}\\
&&\lesssim \|A_{\mathcal S, u}^{m-2}f\|_{L^p(\la^{1-\frac{2}{m}}w^{\frac{2}{m}}u^p)}\lesssim\dots\lesssim \|A_{\mathcal S}f\|_{L^p(w)}\lesssim \|f\|_{L^p(w)}\lesssim \|f\|_{L^p(\mu)},
\end{eqnarray*}
which proves (\ref{ansuf}).

Consider now the assumption $\la, \mu^{1-p'}\in A_{\infty}$. Observe that if $\mu\in A_p$, then $\mu^{1-p'}\in A_{p'}$, and then, by duality, the situation is reduced to the previously considered case
(and we even do not need to use that $\la\in A_{\infty}$). Therefore, assume that $\mu\not\in A_p$.

Combining Corollary \ref{apc} and the fact that $\la$ and $\mu^{1-p'}$ satisfy the reverse H\"older inequality, we obtain that there exists $r>1$ such that $(\la^r,\mu^r)\in A_p$. Therefore, by
Theorem \ref{N}, there exists $\nu\in A_p$ such that $\la\lesssim \nu\lesssim \mu$ a.e. Since $\mu\not\in A_p$, we have that $\frac{\mu}{\nu}\not\in L^{\infty}$.
Therefore, we are in position to repeat the previous argument with $\la$ replaced by $\nu$.
This completes the proof.
\end{proof}

\section{Appendix}
D. Cruz-Uribe and K. Moen \cite{CM}
showed that the condition
\begin{equation}\label{crmcond}
\sup_Q \|u^{1/p}\|_{L^p(\log L)^{2p-1+\e},Q}\|v^{-1/p}\|_{L^{p'}(\log L)^{2p'-1+\d},Q}<\infty
\end{equation}
is sufficient for
\begin{equation}\label{tb1}
\|T_b^1f\|_{L^p(u)}\lesssim \|b\|_{BMO}\|f\|_{L^p(v)},
\end{equation}
and also they showed that this result is not true for $\e=\d=0$.

On the other hand, by Theorem \ref{extbctbm},
\begin{equation}\label{sepb1}
\sup_Q \|u^{1/p}\|_{L^p(\log L)^{p-1+\e},Q}\|v^{-1/p}\|_{L^{p'}(\log L)^{2p'-1+\d},Q}<\infty
\end{equation}
and
\begin{equation}\label{sepb2}
\sup_Q \|u^{1/p}\|_{L^p(\log L)^{2p-1+\e},Q}\|v^{-1/p}\|_{L^{p'}(\log L)^{p'-1+\d},Q}<\infty
\end{equation}
provide another sufficient condition for (\ref{tb1}).

It is obvious that condition (\ref{crmcond}) implies (\ref{sepb1}) and (\ref{sepb2}). We give an example showing that
(\ref{sepb1}) and (\ref{sepb2}) are weaker than (\ref{crmcond}), in general.

\begin{theorem}\label{example}
There exist weights $u$ and $v$ on ${\mathbb R}$ such that
\begin{equation}\label{cl1}
\sup_I \|u^{1/p}\|_{L^p(\log L)^{p-\frac{1}{2}},I}\|v^{-1/p}\|_{L^{p'}(\log L)^{2p'-\frac{1}{2}},I}<\infty
\end{equation}
and
\begin{equation}\label{cl2}
\sup_I \|u^{1/p}\|_{L^p(\log L)^{2p-\frac{1}{2}},I}\|v^{-1/p}\|_{L^{p'}(\log L)^{p'-\frac{1}{2}},I}<\infty,
\end{equation}
while
\begin{equation}\label{cl3}
\sup_I \|u^{1/p}\|_{L^p(\log L)^{2p-1},I}\|v^{-1/p}\|_{L^{p'}(\log L)^{2p'-1},I}=\infty.
\end{equation}
\end{theorem}

\subsection{Auxiliary propositions}
We say that a Young function $\Phi$ is submultiplicative if there exists $\kappa\ge 1$ such that
\begin{equation}\label{sub}
\Phi(ab)\le \kappa\Phi(a)\Phi(b)
\end{equation}
for all $a,b\ge 0$. It is easy to see that the function
$$\f(t)=t\log^{\a}(t+e)\quad(\a\ge 0)$$
is submultiplicative.

In the following propositions we assume that $\Phi$ satisfies (\ref{sub}).

\begin{prop}\label{Lemma:LargerCube} Let $I,J\subset {\mathbb R}$ be the intervals such that $J\subset I$. Then
\[
\|f\|_{\Phi,J}\leq\|f\|_{\Phi,I}\frac{1}{\Phi^{-1}\left(\frac{1}{\kappa}\frac{|J|}{|I|}\right)}.
\]
\end{prop}

\begin{proof}
By (\ref{sub}),
$$\Phi\left(\frac{|f|}{\lambda}\right)\le \frac{|J|}{|I|}\Phi\left(\frac{|f|}{\lambda\Phi^{-1}\left(\frac{1}{\kappa}\frac{|J|}{|I|}\right)}\right).$$
Using also that $J\subset I$, we obtain
$$
\frac{1}{|J|}\int_{J}\Phi\left(\frac{|f|}{\lambda}\right)\le \frac{1}{|I|}\int_{I}
\Phi\left(\frac{|f|}{\lambda\Phi^{-1}\left(\frac{1}{\kappa}\frac{|J|}{|I|}\right)}\right).
$$
Hence if $\lambda=\|f\|_{\Phi,I}\frac{1}{\Phi^{-1}\left(\frac{1}{\kappa}\frac{|J|}{|I|}\right)}$,
the latter is controlled by $1$, and the desired conclusion follows.
\end{proof}

\begin{prop}\label{Lemma:Supp} Let $I,J\subset {\mathbb R}$ be the intervals
such that $|J\cap I|\not=0$. If $\text{supp }f\subset I$, then
\[
\|f\|_{\Phi,J}\leq\|f\|_{\Phi,J\cap I}\frac{1}{\Phi^{-1}\left(\frac{1}{\kappa}\frac{|J|}{|J\cap I|}\right)}.
\]
\end{prop}

\begin{proof} The proof is similar to the previous. By (\ref{sub}),
$$\Phi\left(\frac{|f|}{\lambda}\right)\le \frac{|J|}{|J\cap I|}\Phi\left(\frac{|f|}{\lambda\Phi^{-1}\left(\frac{1}{\kappa}\frac{|J|}{|J\cap I|}\right)}\right).$$
Therefore,
$$
\frac{1}{|J|}\int_{J}\Phi\left(\frac{|f|}{\lambda}\right)\le \frac{1}{|J\cap I|}\int_{J\cap I}
\Phi\left(\frac{|f|}{\lambda\Phi^{-1}\left(\frac{1}{\kappa}\frac{|J|}{|J\cap I|}\right)}\right).
$$
Hence if $\lambda=\|f\|_{\Phi,J\cap I}\frac{1}{\Phi^{-1}\left(\frac{1}{\kappa}\frac{|J|}{|J\cap I|}\right)}$,
the latter is controlled by $1$, and the desired conclusion follows.
\end{proof}

\subsection{Localized weights}
Let $0<a<b<1/2$ be the numbers such that
$$b=\log^{-\frac{1}{2}}(1/a).$$
We also assume that $a$ is small enough so that
\begin{equation}\label{conda}
a<(b/2)^{\max(p-1,1)}.
\end{equation}

Let $m$ any real number. We define two functions on the interval
$I=[m,m+b]$ as follows

\begin{equation}
u_{I}(x)=\frac{1}{a}\chi_{[m,m+a]}(x)+a\chi_{[m+b-a,m+b]}(x)\label{def1}
\end{equation}
and
\begin{equation}
v_{I}(x)=\frac{1}{a}\chi_{[m,m+b-a^{\frac{1}{p-1}})}(x)+a\log^{3p}(1/a)\chi_{[m+b-a^{\frac{1}{p-1}},m+b]}(x).\label{def2}
\end{equation}
Observe that, by (\ref{conda}),
$$m+a<\min(m+b-a, m+b-a^{\frac{1}{p-1}}).$$

\begin{prop}\label{calc}
We have
\begin{equation}\label{b1}
\|u_I^{1/p}\|_{L^p(\log L)^{p-\frac{1}{2}},I}\|v_I^{-1/p}\|_{L^{p'}(\log L)^{2p'-\frac{1}{2}},I}\lesssim 1,
\end{equation}
\begin{equation}\label{b2}
\|u_I^{1/p}\|_{L^p(\log L)^{2p-\frac{1}{2}},I}\|v_I^{-1/p}\|_{L^{p'}(\log L)^{p'-\frac{1}{2}},I}\lesssim 1
\end{equation}
and
\begin{equation}\label{b3}
\|u_I^{1/p}\|_{L^p(\log L)^{2p-1},I}\|v_I^{-1/p}\|_{L^{p'}(\log L)^{2p'-1},I}\sim \frac{1}{b}.
\end{equation}
\end{prop}

\begin{proof} The proof is based on the observation that for any $\a,\b>0$
$$
\|u_I^{1/p}\|_{L^p(\log L)^{\a},I}\sim \|u_I\|_{L(\log L)^{\a},I}^{1/p}
\quad\text{and}\quad
\|v_I^{-1/p}\|_{L^{p'}(\log L)^{\b},I}\sim \|v_I^{1-p'}\|_{L(\log L)^{\b},I}^{1/p'},
$$
and on a straightforward computation by (\ref{eqlog}).

We have
$$
(u_I)_I=\frac{1}{b}(1+a^2)\quad\text{and}\quad (v_I^{1-p'})_I=\frac{1}{b}\Big(
(b-a^{\frac{p'}{p}})a^{\frac{p'}{p}}+\log^{3p(1-p')}(1/a)\Big).
$$

Therefore,  by (\ref{eqlog}),
\begin{eqnarray*}
&&\|u_I\|_{L(\log L)^{\a},I}\sim \frac{1}{|I|}\int_Iu_I\log^{\a}\left(\frac{u_I}{(u_I)_I}+e\right)\\
&&=\frac{1}{b}\log^{\a}\left(\frac{1}{(u_I)_I} \frac{1}{a}+e\right)+\frac{a^2}{b}\log^{\a}\left(\frac{a}{(u_I)_I}+e\right)\\
&&\sim\frac{1}{b}\log^{\a}\left(\frac{b}{a}+e\right)\sim\frac{1}{b}\log^{\a}(1/a).
\end{eqnarray*}

Similarly,
\begin{eqnarray*}
&&\|v_I^{1-p'}\|_{L(\log L)^{\b},I}\sim \frac{1}{|I|}\int_Iv_I^{1-p'}\log^{\b}\left(\frac{v_I^{1-p'}}{(v_I^{1-p'})_I}+e\right)\\
&&=\frac{1}{b}(b-a^{\frac{1}{p-1}})a^{\frac{1}{p-1}}\log^{\b}\left(\frac{a^{\frac{1}{p-1}}}{(v_I^{1-p'})_I}+e\right)\\
&&+\frac{1}{b}\log^{\ga_2(1-p')}(1/a)
\log^{\b}\left(\frac{a^{1-p'}\log^{3p(1-p')}(1/a)}{(v_I^{1-p'})_I}+e\right)\\
&&\sim
\frac{1}{b}\log^{3p(1-p')}(1/a)
\log^{\b}\left(ba^{1-p'}+e\right)\sim \frac{1}{b}\log^{\b-3p(p'-1)}(1/a).
\end{eqnarray*}

From this,
\begin{eqnarray*}
\|u_I^{1/p}\|_{L^p(\log L)^{\a},I}\|v_I^{-1/p}\|_{L^{p'}(\log L)^{\b},I}&\sim&\frac{1}{b}
\log^{\frac{\a}{p}+\frac{\b-3p(p'-1)}{p'}}(1/a)\\
&=&\frac{1}{b}\log^{\frac{\a}{p}+\frac{\b}{p'}-3}(1/a),
\end{eqnarray*}
which implies the proposition.
\end{proof}

\subsection{Proof of Theorem \ref{example}}
For each $n\ge N$, where $N$ is large enough,  set $a_{n}=e^{-n}$, $b_n=n^{-1/2}$ and
$m_n=e^{n}$. Set $I_{n}=[m_{n},m_{n}+b_{n}]$, and define $u_{n}(x)=u_{I_{n}}(x)$ and $v_{n}(x)=v_{I_{n}}(x)$
as described in the previous section.

Finally, set
\[
u(x)=\sum_{n=N}^{\infty}u_{n}(x)
\]
and
\[
v(x)=\chi_{(-\infty,e^{N})}+\sum_{n=N}^{\infty}v_{n}(x)+\sum_{n=N}^{\infty}e^{n}\chi_{(e^{n}+b_{n},e^{n+1})}(x).
\]

By (\ref{b3}),
$$
\|u^{1/p}\|_{L^p(\log L)^{2p-1},I_n}\|v^{-1/p}\|_{L^{p'}(\log L)^{2p'-1},I_n}\sim n^{1/2},
$$
which proves (\ref{cl3}).

It remains to show that (\ref{cl1}) and (\ref{cl2}) hold, namely that for every interval $J\subset {\mathbb R}$,
\begin{equation}\label{ch1}
\|u^{1/p}\|_{L^p(\log L)^{p-\frac{1}{2}},J}\|v^{-1/p}\|_{L^{p'}(\log L)^{2p'-\frac{1}{2}},J}\lesssim 1
\end{equation}
and
\begin{equation}\label{ch2}
\|u^{1/p}\|_{L^p(\log L)^{2p-\frac{1}{2}},J}\|v^{-1/p}\|_{L^{p'}(\log L)^{p'-\frac{1}{2}},J}\lesssim 1.
\end{equation}

We shall consider several subcases.

\vskip 0.3 cm
{\bf Case 1:} $J\subset I_{n}$ for some $n$. We split $I_{n}$ as
\[
I_{n}=I^{1}\cup I^{2}\cup I^{3},
\]
where
\[
I^{1}=[m,m+a_{n}),\quad I^{2}=[m+a_{n},m+b_{n}-a_{n}^{\frac{1}{p-1}})\quad\text{and}\quad I^{3}=[m+b_{n}-a_{n}^{\frac{1}{p-1}},m+b_{n}].
\]
There are four possible situations:
\begin{enumerate}
\renewcommand{\labelenumi}{(\roman{enumi})}
\item $J\subset I^{i}$ for some $i$.
\begin{itemize}
\item If $J\subset I^{1}$, then, since $u$ and $v$ are constant on $I^1$,
\[
\|u^{1/p}\|_{L^p(\log L)^{p-1/2},J}=\|u^{1/p}\|_{L^p(\log L)^{2p-1/2},J}=a_{n}^{-1/p}
\]
and
\begin{equation}\label{vcons}
\|v^{-1/p}\|_{L^{p'}(\log L)^{2p'-1/2},J}=\|v^{-1/p}\|_{L^{p'}(\log L)^{p'-1/2},J}=a_{n}^{1/p}.
\end{equation}
If $J\subset I^{3},$ then, since $u\le a_n$ on $I^3$,
\begin{equation}\label{an}
\|u^{1/p}\|_{L^p(\log L)^{p-1/2},J}=\|u^{1/p}\|_{L^p(\log L)^{2p-1/2},J}\le a_{n}^{1/p}.
\end{equation}
Also, since $v$ is constant on $I^3$,
\[
\|v^{-1/p}\|_{L^{p'}(\log L)^{2p'-1/2},J}=\|v^{-1/p}\|_{L^{p'}(\log L)^{p'-1/2},J}=(a_n\log^{3p}(1/a_n))^{-1/p}.
\]
In both cases (\ref{ch1}) and (\ref{ch2}) hold.

\item Suppose that $J\subset I^{2}$. If $p'\geq p$ then $u=0$ on $J$. Otherwise $u\le a_n$ on $J$,
and hence, in both cases (\ref{an}) holds.
Also,
$$
\|v^{-1/p}\|_{L^{p'}(\log L)^{2p'-1/2},J}=\|v^{-1/p}\|_{L^{p'}(\log L)^{p'-1/2},J}=a_{n}^{1/p},
$$
and hence we again obtain (\ref{ch1}) and (\ref{ch2}).
\end{itemize}

\vskip 0.3 cm
\item $J\cap I^{i}\not=\emptyset$ for every $i$. Then $|J|\sim |I_n|$. In this case
(\ref{ch1}) and (\ref{ch2}) follow from a combination of Propositions \ref{Lemma:LargerCube} and \ref{calc}.

\vskip 0.3 cm
\item $J\cap I^{1}\not=\emptyset$, $J\cap I^{2}\not=\emptyset$ and $J\cap I^{3}=\emptyset$.
Then (\ref{vcons}) holds, and also, by Proposition \ref{Lemma:Supp},
$$
\max(\|u^{1/p}\|_{L^p(\log L)^{p-1/2},J},\|u^{1/p}\|_{L^p(\log L)^{2p-1/2},J})\lesssim a_{n}^{-1/p},
$$
which implies (\ref{ch1}) and (\ref{ch2}).

\vskip 0.3 cm
\item $J\cap I^{1}=\emptyset$, $J\cap I^{2}\not=\emptyset$ and $J\cap I^{3}\not=\emptyset$.
Then, arguing as above,
$$
\max(\|u^{1/p}\|_{L^p(\log L)^{p-1/2},J},\|u^{1/p}\|_{L^p(\log L)^{2p-1/2},J})\lesssim a_n^{1/p}
$$
and
$$
\max(\|v^{-1/p}\|_{L^{p'}(\log L)^{2p'-1/2},J},\|v^{-1/p}\|_{L^{p'}(\log L)^{p'-1/2},J})\lesssim (a_n\log^{3p}(1/a_n))^{-1/p},
$$
and therefore, (\ref{ch1}) and (\ref{ch2}) hold.
\end{enumerate}

\vskip 0.3cm
{\bf Case 2:} $J\cap I_{n}\protect\not=\emptyset$ for just one $I_{n}$ but $J\protect\not\subset I_{n}$.
In this case, by Proposition \ref{Lemma:Supp},
$$
\|u^{\frac{1}{p}}\|_{L^p(\log L)^{p-1/2},J}\lesssim\|u^{\frac{1}{p}}\|_{L^p(\log L)^{p-1/2},J\cap I_{n}}
$$
and
$$
\|u^{\frac{1}{p}}\|_{L^p(\log L)^{2p-1/2},J}\lesssim\|u^{\frac{1}{p}}\|_{L^p(\log L)^{2p-1/2},J\cap I_{n}}.
$$

On the other hand we note that for any $x,y\in J$ with $x\in J\setminus I_{n}$
and $y\in I_{n}$ $v^{-\frac{1}{p}}(x)\lesssim v^{-\frac{1}{p}}(y)$.
Hence,
$$
\|v^{-1/p}\|_{L^{p'}(\log L)^{2p'-1/2},J}\lesssim \|v^{-1/p}\|_{L^{p'}(\log L)^{2p'-1/2},J\cap I_n}
$$
and
$$
\|v^{-1/p}\|_{L^{p'}(\log L)^{p'-1/2},J}\lesssim \|v^{-1/p}\|_{L^{p'}(\log L)^{p'-1/2},J\cap I_n}.
$$
This reduces us to the previous case and hence we are done.

\vskip 0.3cm
{\bf Case 3:} $J\cap I_{n}\protect\not=\emptyset$ for more than one $I_{n}$. Using that
$$\|f\|_{L(\log L)^{\a},Q}\lesssim \|f\|_{L^r,Q}$$
for $r>1$, it suffices to show that for $r>1$ small enough
$$\|u\|_{L^r,J}\lesssim 1\quad\text{and}\quad \|v^{1-p'}\|_{L^r,J}\lesssim 1.$$

Let $n_{0}$ and $n_1$ be the smallest and the largest integers such that
$J\cap I_{n_{0}}\not=\emptyset$ and $J\cap I_{n_{1}}\not=\emptyset$. If $n_0>N$, then $J\subset [e^N,\infty)$ and
$|J|\sim e^{n_1}$. If $n_0=N$, then one can write $J=L\cup R$, where $L\subset (-\infty, e^N)$ (possibly $L=\emptyset$)
and $R\subset [e^N,\infty)$ with $|R|\sim e^{n_1}$.

Since $u$ is supported in $\cup_{n\ge N}I_n$, we obtain that
$$\int_Ju^2\le \sum_{n=n_0}^{n_1}\Big(e^n+e^{-3n}\Big)\lesssim e^{n_1}\lesssim |J|.$$

Now, if $n_0>N$, then for $r\le p$,
$$\int_Jv^{(1-p')r}\lesssim \sum_{n=n_0-1}^{n_1+1}\Big(e^{nr(1-p')}e^n+\Big(\frac{n^{3p}}{e^n}\Big)^{r(1-p')}e^{-n(p'-1)}\Big)\lesssim e^{n_1}\lesssim |J|.$$
In the case if $n_0=N$ and $J=L\cup R$ represented as above, then
$$\int_Jv^{(1-p')r}=|L|+\int_Rv^{(1-p')r}\lesssim |L|+|R|=|J|.$$
This completes the proof.

\end{document}